\pdfoutput=1

\documentclass{amsart}
\usepackage[centertags]{amsmath}
\usepackage{amsfonts}
\usepackage{amssymb}
\usepackage{amsthm}
\usepackage{ subfigure}
\usepackage{graphicx}
\usepackage{fullpage }
\usepackage{color}

\newtheorem{theorem}{Theorem}

\newtheorem{lemma}[theorem]{Lemma}

\newtheorem{corollary}[theorem]{Corollary}
\theoremstyle{definition}
\newtheorem{definition}[theorem]{Definition}

\theoremstyle{remark}

\newtheorem{example}[theorem]{Example}

\begin{document}

\renewcommand{\labelenumii}{(\roman{enumii})}

\def\printname#1{
	
		\smash{\makebox[0pt]{\hspace{-0.5in}
			\raisebox{8pt}{\tt\tiny #1}}}
}
\def\lbl#1{\label{#1}
\printname{#1}}

\newcommand{\Z}{\mathbb{Z}}
\newcommand{\F}{\mathcal{F}}
\newcommand{\calL}{\mathcal{L}}
\newcommand{\bbL}{\mathbb{L}}
\newcommand{\V}{\mathcal{V}}
\newcommand{\E}{\mathcal{E}}
\newcommand{\M}{\mathcal{M}}

\newcommand{\ar}[1]{\overrightarrow{#1}}

\newcommand{\BG}{{\bf G}}
\newcommand{\bv}{{\bf v}}
\newcommand{\BV}{{\bf V}}
\newcommand{\BE}{{\bf E}}
\newcommand{\BA}{{\bf A}}
\newcommand{\be}{{\bf e}}
\newcommand{\BB}{{\bf B}}
\newcommand{\BF}{{\bf H}}
\newcommand{\ba}{{\bf a}}
\newcommand{\bb}{{\bf b}}
\newcommand{\bu}{{\bf u}}

\newcommand{\p}{{\prime}}



\title[ A characterization of partially dual graphs ]{ A characterization of partially dual graphs}

\author[I.~Moffatt]{Iain Moffatt$^*$}

\begin{abstract}
In this paper, we extend the recently introduced  concept of partially dual ribbon graphs  to graphs. We then go on to  characterize  partial duality of graphs in terms of bijections between edge sets of corresponding graphs. This result generalizes a well known result of J.~Edmonds in which natural duality of graphs is characterized in terms of edge correspondence, and gives a combinatorial characterization of partial duality. 

\end{abstract}

\thanks{
${\hspace{-1ex}}^*$
Department of Mathematics and Statistics,  University of South Alabama, Mobile, AL 36688, USA; \\
${\hspace{.35cm}}$ \includegraphics[width=43mm]{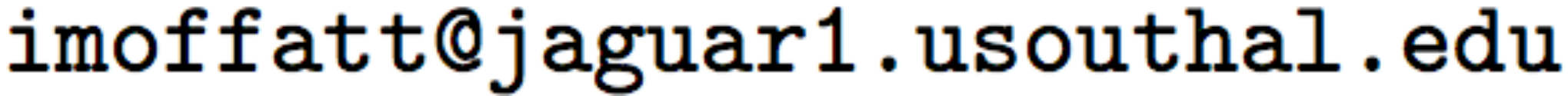}}

\date{\today}

\maketitle

\section{Introduction and motivation}\label{s.intro}
S.~Chmutov recently introduced the concept of the partial dual $\BG^{\BA}$of a ribbon graph $\BG$ (\cite{Ch1}). Partial duality generalizes the  natural dual (or Euler-Poincar\'e dual or geometric dual) of a ribbon graph by forming the dual of $\BG$ only with respect to a subset of its edges $\BA$ (a formal definition of partial duality is given in Section~\ref{s.pd}). 
In contrast with natural duality, where the topologies  of $\BG$ and $\BG^*$ are similar, the topology of a partial dual $\BG^{\BA}$ can be very different from the topology of $\BG$.
For example, 
although a ribbon graph and its natural dual always have the same genus, a ribbon graph and a partial dual need not.

As one would expect with a generalization of duality, partial duality has desirable properties. For example, (up to normalization and specialization) the weighted (Bollob\'{a}s-Riordan) ribbon graph polynomials of $\BG$ and $\BG^{\BA}$ are equal (see \cite{Ch1, Mo2} and \cite{VT} for a multivariate version). This generalizes the well known relation between the Tutte polynomial of a plane graph and its natural dual: $T(G;x,y)=T(G^*;y,x)$.

One particularly significant application of partial duality is to be found in knot theory. Recently  
there has been a lot of interest in the connection between knots, knot invariants, ribbon graphs and ribbon graph polynomials (\cite{CP,CV,Ch1,Da,Mo1,Mo2,Mo3}). Partial duality provides a way to connect these recent results with each other (see \cite{Mo2}, where the ``unsigning'' process is a special case of partial duality, and \cite{Ch1}). 

We would expect many other properties of the natural dual to extend  to partial duality. Here we are interested in generalizing a theorem of J.~Edmonds from \cite{Ed} by finding a characterization of partial duality in terms of a bijection between edge sets. Edmonds' characterization of dual graphs generalizes Whitney's well known characterization of planar graphs in terms of (combinatorial) duals from \cite{Wh}. It reads  as follows:
\begin{theorem}[Edmonds \cite{Ed}]
A 1-1 correspondence between the edges of two connected graphs is a duality with respect to some polyhedral surface embedding if and only if for each vertex $v$ of each graph, the edges which meet $v$  correspond in the other graph  to the edges of  a subgraph $G_v$ which is Eulerian. That is $G_v$ is connected and has an even number of edge-ends to each of its vertices (where if an edge meets $v$ at both ends its image in $G_v$ is counted twice.)
\end{theorem}
We will refer to the conditions that the bijection in Edmonds' theorem  satisfies  as {\em Edmonds' Criteria}.

We will say that two graphs are partial duals if they are the cores of partially dual ribbon graphs. (The core of a ribbon graph as defined in Subsection~\ref{ss.nd}. Essentially the core of a ribbon graph is the graph obtained by ``forgetting the topological structure'' of a ribbon graph.) Suppose that $G$ and $G^A$ are partial dual graphs and $\varphi(A)$ is the set of edges of $G^A$ that correspond with the set $A$. If  $A^c:=\E(G)\backslash A$ and  $\varphi(A^c)=\E(G^A)\backslash \varphi(A)$,  it turns out that    the graphs $G\backslash A^c$ and $G^A\backslash \varphi(A^c)$ are natural duals and can therefore be dually embedded in a surface $\Sigma$.  This dual embedding does not record any information about the edges in $A^c$ or $\varphi(A^c)$. However, we will see that all of the information about these two sets of edges can be recorded by placing a set of embedded edges between the vertices of the dual embeddings of $G\backslash A^c$ and $G^A\backslash \varphi(A^c)$. Moreover, we will see that  such a structure characterizes partially dual graphs.  

Our extension of Edmonds' theorem follows from the fact that partial duality is characterized by the existence of two dually embedded graphs that are  decorated with edges in a certain way. Within this structure the dually embedded graphs are characterized by a bijection satisfying Edmonds' criteria.   We can extend this bijection so that it also describes the decorating edges, to obtain a characterization of partial duality in terms of a bijection between edge sets. This gives us our main theorem which reads as follows:
\begin{theorem}[Main Theorem.]
Two graphs $G$ and $H$ are partial duals if and only if there exists a bijection $\varphi: \E(G) \rightarrow \E(H)$, such that 
\begin{enumerate}
\item $\left. \varphi\right|_{A}:  A\rightarrow \varphi(A)$ satisfies Edmonds' Criteria for some subset $A\subseteq \E(G)$;
\item If $v\in \V(G)$ is incident to an edge in $A$, and if $e\in \E(G)$ is incident to $v$, then $\varphi(e)$ is incident to a vertex of $\varphi(A)_v$. Moreover, if both ends of $e$ are incident to $v$, then both ends of $\varphi(e)$ are incident to vertices of $\varphi(A)_v$.

\item If $v\in \V(G)$ is not incident to an edge in $A$, then there exists a vertex $v' \in \V(H)$ with the property that  $e\in \E(G)$ is incident to $v$ if and only if  $\varphi(e)\in \E(H)$ is incident to $v'$. Moreover, both ends of $e$ are incident to $v$ if and only if both ends of $\varphi(e)$ are incident to $v'$.

\end{enumerate}
Here $\varphi(A)_v $ is the subgraph of $H$ induced by the images of the edges from $A$ that are incident with $v$. 
\end{theorem}
 This theorem appears in Subsection~\ref{ss.ged}, where a proof is given, as Theorem~\ref{t.main}.   The theorem  answers Chmutov's Problem~3 from Section~7 of \cite{Ch1} where it was asked if Edmonds' Theorem could be extended to partial duality.

\smallskip

The paper is structured as follows. We give the definition of a ribbon graph and provide two combinatorial presentations of ribbon graphs in Section~\ref{s.rg}. Section~\ref{s.pd} examines the relation between  partial and natural duality. In particular a description of partial dual ribbon graphs as a pair of embedded, marked, naturally dual ribbon graphs is given. This description of partial duals is used in  Section~\ref{s.pdg} to give a characterization of partial dual graphs as a pair of decorated, dually embedded graphs. This structure is described in terms of a bijection between edge sets of graphs in  Subsection~\ref{ss.ged}, giving our generalization of Edmonds' Theorem.

\smallskip

I am indebted to the anonymous referee whose suggestions greatly improved the statement of this paper's  main theorem.


\section{Ribbon graphs}\label{s.rg}
 In this short section we define ribbon graphs and provide two other combinatorial descriptions of ribbon graphs that we will need. We emphasize the fact that our ribbon graphs may be non-orientable, and therefore contain more information than combinatorial maps. 

\begin{definition}
A {\em ribbon graph} $\BG =\left(  \V(\BG),\E(\BG)  \right)$ is a (possibly non-orientable) surface with boundary represented as the union of two  sets of topological discs: a set $\V (\BG)$ of {\em vertices}, and a set of {\em edges} $\E (\BG)$ such that 
\begin{enumerate}\renewcommand{\labelenumi}{(\roman{enumi})}
\item the vertices and edges intersect in disjoint line segments;
\item each such line segment lies on the boundary of precisely one
vertex and precisely one edge;
\item every edge contains exactly two such line segments.
\end{enumerate}
\end{definition}

Ribbon graphs are considered up to  homeomorphisms of the surface that preserve the vertex-edge structure.  
The embedding of a ribbon graph in three-space is irrelevant. 

 \begin{figure}
\begin{tabular}{ccccc}
\includegraphics[width=40mm]{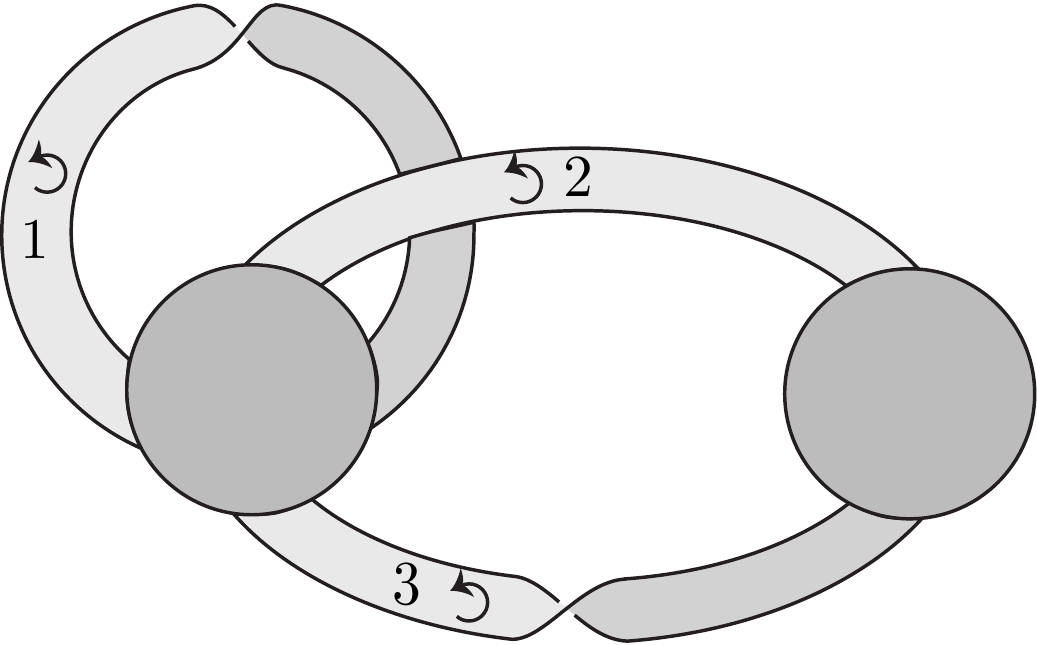} &\raisebox{9mm}{ = }& \includegraphics[width=40mm]{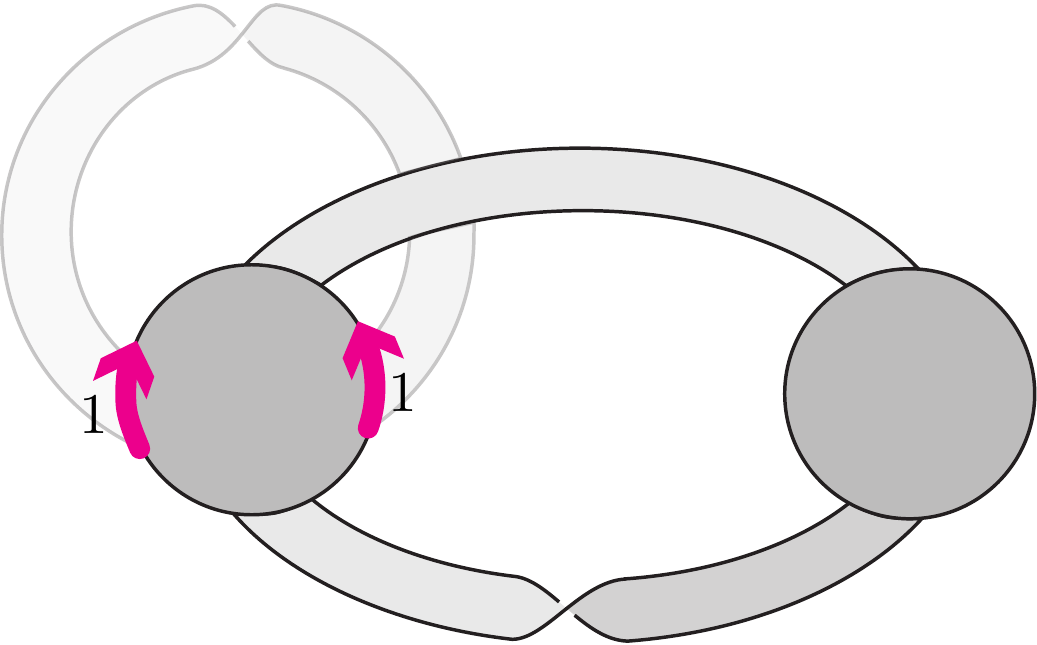} &\raisebox{9mm}{ = }& \includegraphics[width=40mm]{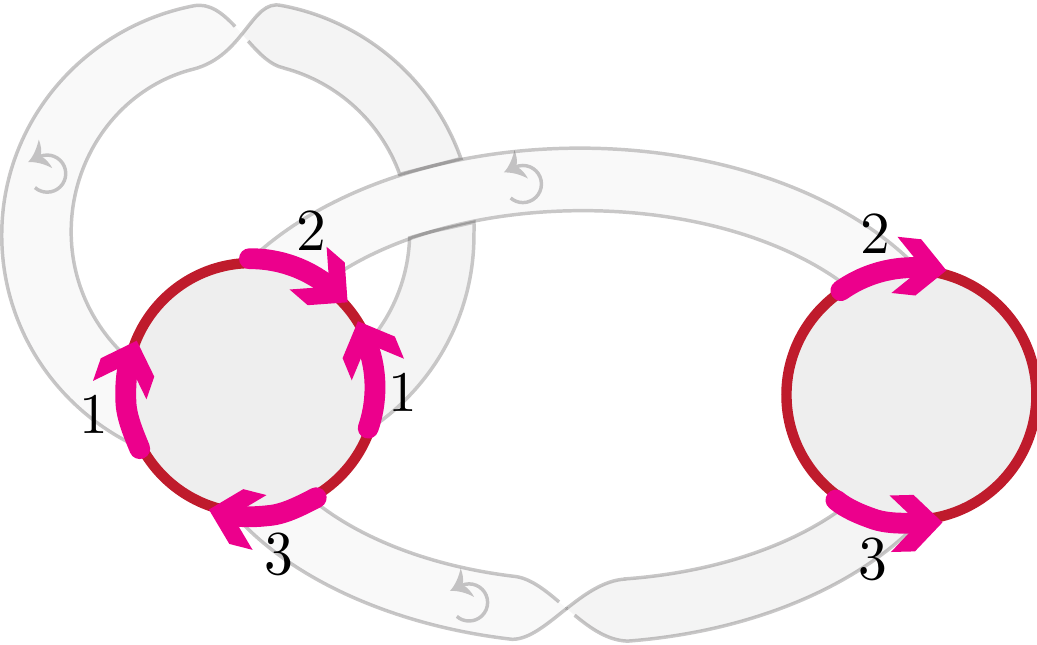} \\
(i)&&(ii)& &(iii)  
\end{tabular}
\caption{Realizations of a ribbon graph}
\label{fig.rgex}
\end{figure}

We will always assume that there is some distinct labelling of the edges in the set $\E(\BG)$  of edges  of a ribbon graph $\BG$. This allows us to abuse language and say that  we  colour an object by $\be$ in $\E(\BG)$, when what we mean that we colour the object with the unique label of $\be$.

\bigskip

It will be convenient to use a description of a ribbon graph $\BG$ as a spanning sub-ribbon graph equipped with a set of coloured arrows that record where the missing edges of the ribbon graph were. A  {\em spanning sub-ribbon graph} of $\BG$ is a ribbon graph $\BF$ which can be obtained from $\BG$ by deleting some edges.

\begin{definition}
An {\em arrow-marked ribbon graph} $\ar{\BG}$ consists of a ribbon graph $\BG$ equipped with a collection of  coloured arrows, called {\em marking arrows}, on the boundaries of its vertices. The marking arrows are such that no marking arrow meets an edge of the ribbon graph, and    there are exactly two marking arrows of each colour.
\end{definition}

Arrow-marked ribbon graphs are considered equivalent if one can be obtained from the other by reversing the direction of all of the marking arrows which belong to some subset of colours.

A ribbon graph can be obtained from an arrow-marked ribbon graph by adding edges in a way prescribed by the marking arrows, thus: take a disc (this disc will form the new edge) and orient its boundary arbitrarily. Add this disc to the ribbon graph by choosing two non-intersecting arcs on the boundary of the disc and two marking arrows of the same colour, and then identifying the arcs with the marking arrows according to the orientation of the arrow. The disc that has been added forms an edge of a new ribbon graph. 
This process is illustrated in the diagram below, and an example of an arrow-marked ribbon graph and the ribbon graph it describes  is given in Figures~\ref{fig.rgex}(i) and (ii).
\[\includegraphics[height=15mm]{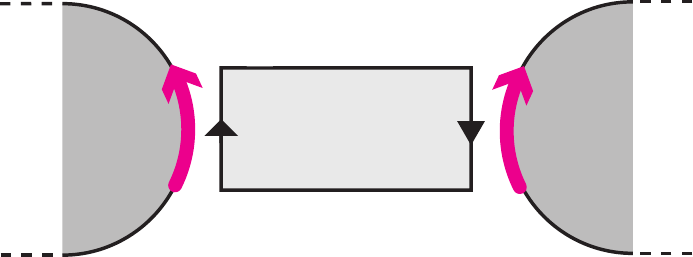} 
\raisebox{6mm}{\hspace{3mm}\includegraphics[width=12mm]{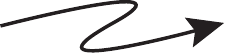}\hspace{3mm}}
  \includegraphics[height=15mm]{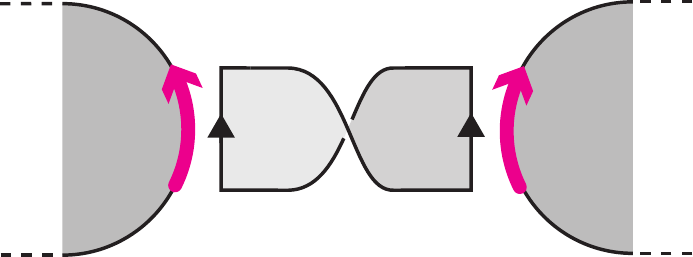} 
  \raisebox{6mm}{\hspace{3mm}\includegraphics[width=12mm]{arrow}\hspace{3mm}} \includegraphics[height=15mm]{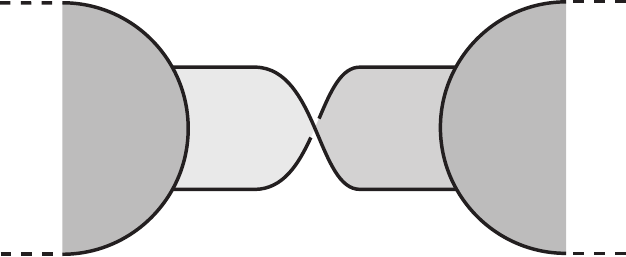} \]

By the above, we know that an arrow-marked ribbon graph describes a ribbon graph. Conversely, every ribbon graph can be described as an arrow-marked spanning sub-ribbon graph. 
To see why this is, suppose that $\BG$ is a ribbon graph and $\BB \subseteq \E(\BG)$.  To describe $\BG$ as an arrow-marked ribbon graph $\ar{\BG\backslash \BB}$, start by arbitrarily orienting each edge in $\BB$. This induces an orientation on the boundary of each edge in $\BB$. To construct the marking arrows: for each $\be\in \BB$, place an arrow on each of the two arcs where $\be$ meets vertices of $\BG$, the direction of this arrow should follow the orientation of the boundary of $\be$; colour the two arrows with $\be$; and delete the edge $\be$. This gives a marked ribbon graph $\ar{\BG\backslash \BB}$. Moreover, the original ribbon graph $\BG$ can be recovered from $\ar{\BG\backslash \BB}$ by adding edges to $\ar{\BG\backslash \BB}$ as prescribed by the marking arrows. 

Notice that if $\BG$ is a ribbon graph and $\BF$ is any spanning sub-ribbon graph, then there is an arrow-marked ribbon graph  of $\ar{\BF}$ which describes  $\BG$.

\bigskip

Every  ribbon graph $\BG$ has a representation as an arrow-marked ribbon graph $\ar{\V(\BG)}$, where the spanning sub-ribbon graph consists of the vertex set  of $\BG$.   In such cases, to describe $\BG$ it is enough to record only the marked boundary cycles of the vertex set (to recover the vertex set, just place each cycle on the boundary of a disc).  Thus a ribbon graph can be presented as a set of cycles with marking arrows on them. In such a structure, there are exactly two marking arrows of each colour. 
Such a structure is called an {\em arrow presentation}. A ribbon graph can be recovered from an arrow presentation by regarding the marked cycles as boundaries of discs, giving an arrow-marked ribbon graph. To describe this more formally:
\begin{definition}
An  {\em arrow presentation} of a ribbon graph consists of a set of oriented (topological) circles (called {\em cycles}) that are marked with coloured arrows, called {\em marking arrows}, such that there are exactly two marking arrows of each colour.   
\end{definition} 
An example of a ribbon graph and its arrow presentation is given in Figure~\ref{fig.rgex}(i) and (iii).

Two arrow presentations are considered equivalent if one can be obtained from the other by reversing pairs of marking arrows of the same colour.

\section{Partial duality}\label{s.pd}
As mentioned above, partial duality is a generalization of the natural dual of a ribbon graph. 
A key feature of partial duality is that it provides a  way extend the well known relation $T(G;x,y)=T(G^*;y,x)$, relating the Tutte polynomial of a planar graph and its dual, to the weighted ribbon graph polynomial. This extension is of interest to knot theorists as it provides a unification of recent results relating the Jones polynomial and Bollob\'as and Riordan's ribbon graph polynomial.

In this section we give a definition of partial duality and then go on to discuss the relationship between partial duals and naturally dual arrow-marked ribbon graphs. This gives rise to the notion of a partial dual embedding of ribbon graphs, an idea that will play a key role in our generalization of Edmonds' Theorem in Section~\ref{s.pdg}.

\subsection{Partial duality}

Although the construction of the partial dual $\BG^{\BA}$ of $\BG$ is perhaps a little lengthy to write down, in practice the formation of the partial dual is a straightforward process.

\begin{definition}
Let $\BG$ be a ribbon graph and $\BA \subseteq \E(\BG)$. The {\em partial dual} $\BG^{\BA}$ of $\BG$ along $\BA$ is defined below. (The construction is shown locally at an edge $\be$ in Figure~\ref{fig.pdd}.)
\begin{enumerate}\renewcommand{\theenumi}{Step P\arabic{enumi}.}
\item Give every edge in $\E(\BG)$ an orientation (this need not extend to an orientation of the whole ribbon graph). Construct a set of marked, oriented, disjoint paths on the boundary of the edges of $\BG$ in the following way:
\begin{enumerate}
\item If $\be\notin \BA$ then the intersection of the edge $\be$ with its incident vertices (or vertex if $\be$ is a loop) defines the two paths. Mark  each of these paths with an arrow which points in the direction of the orientation of the boundary of the edge.  Colour both of these marks with $\be$ .
\item If $\be\in \BA$ then the two sides of $\be$ which do not meet the vertices define the two paths.  Mark  each of these paths with an arrow which points in the direction of the orientation of the boundary of the edge.  Colour both of these marks with $\be$ .
\end{enumerate}
\begin{figure}
\begin{tabular}{ccccc}
 \includegraphics[width=3cm]{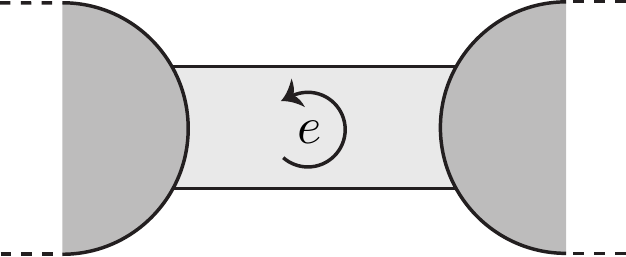} & \hspace{1cm} & \includegraphics[width=3cm]{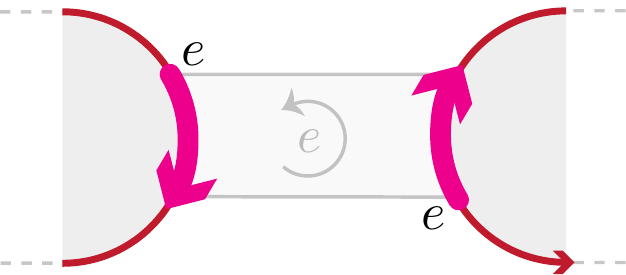} & \hspace{1cm} &   \includegraphics[width=3cm]{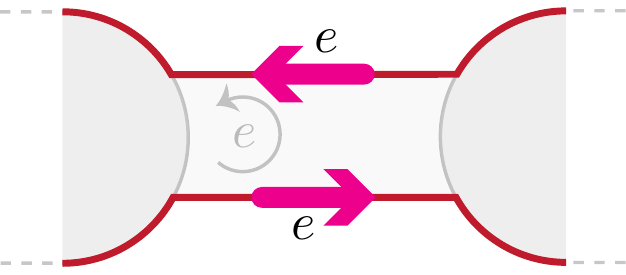}  \\ & &&& \\
 An untwisted edge $\be$. &  &If $\be\notin \BA$. & & If $\be\in \BA$. \\
  \includegraphics[width=3cm]{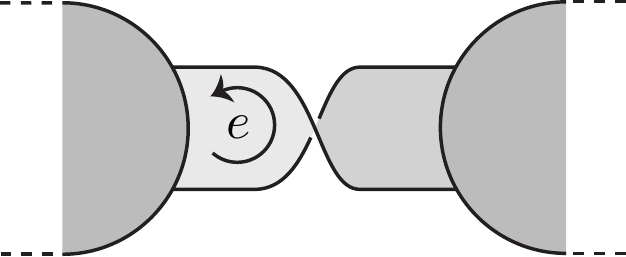} & \hspace{1cm} & \includegraphics[width=3cm]{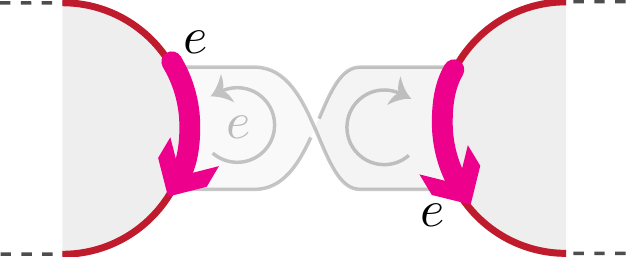} & \hspace{1cm} &   \includegraphics[width=3cm]{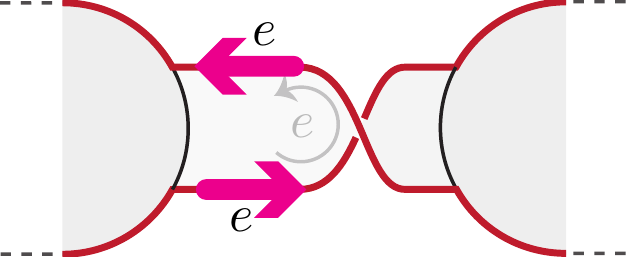}  \\ & &&& \\
A twisted edge $\be$. &  &If $\be\notin \BA$. & & If $\be\in \BA$.
 \end{tabular}
\caption{Forming paths in the partial dual.}
\label{fig.pdd}
\end{figure}

\item Construct a set of closed curves on the boundary of $\BG\backslash \BA^c$ by joining the marked paths constructed above by connecting them along the boundaries of $\BG\backslash \BA^c$ in the natural way.

\item This defines a collection of  non-intersecting, closed curves on the boundary of $\BG\backslash \BA^c$ which are marked with coloured, oriented arrows. This is precisely an arrow presentation of a ribbon graph. The corresponding ribbon graph is the partial dual $\BG^{\BA}$.    

\end{enumerate}
\end{definition}

Two examples of the construction of a partial dual are shown below.
\begin{example}\label{ex.dex1} ~

 \begin{center}
\begin{tabular}{ccccc}
 \includegraphics[width=4cm]{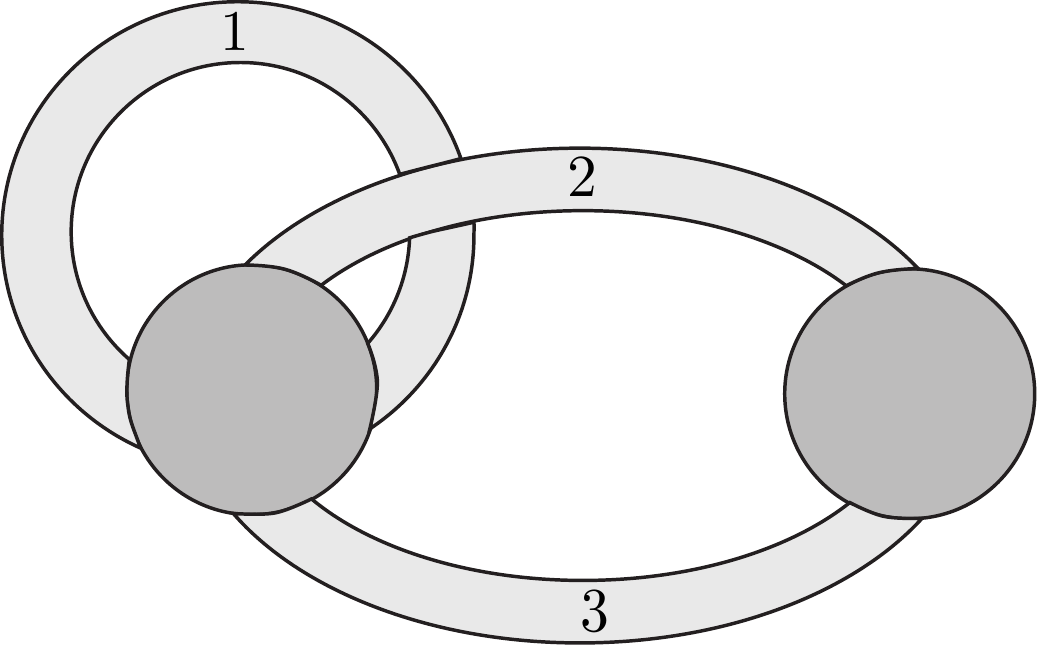}  &\hspace{3cm}& \includegraphics[width=4cm]{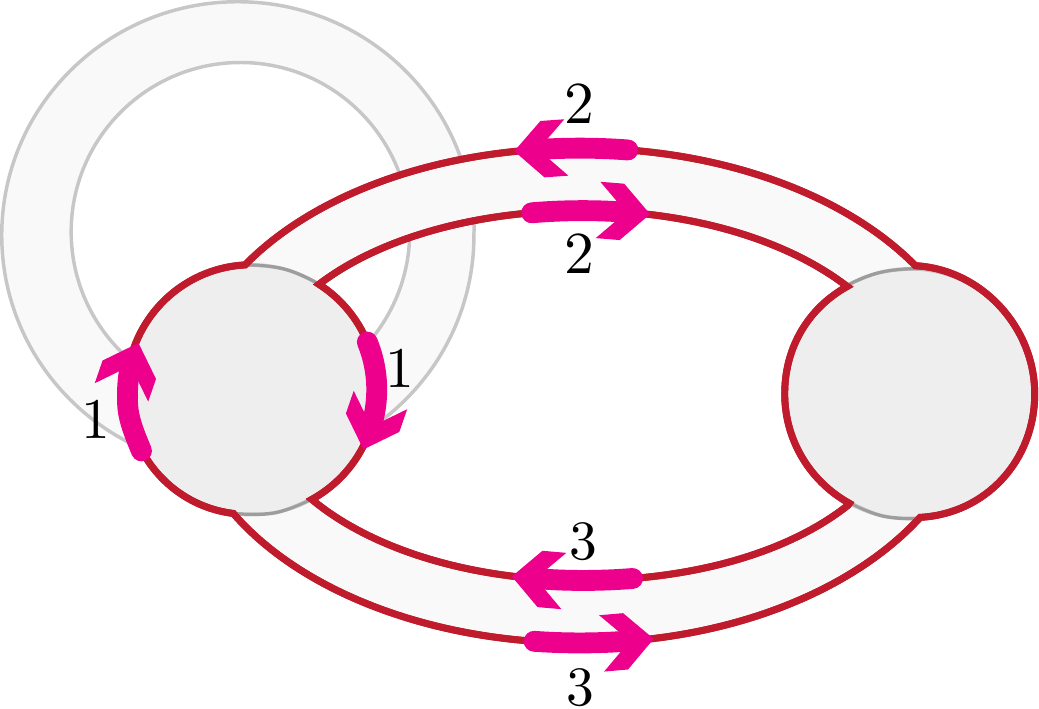}   &\hspace{3cm}&  \raisebox{5mm}{  \includegraphics[width=4cm]{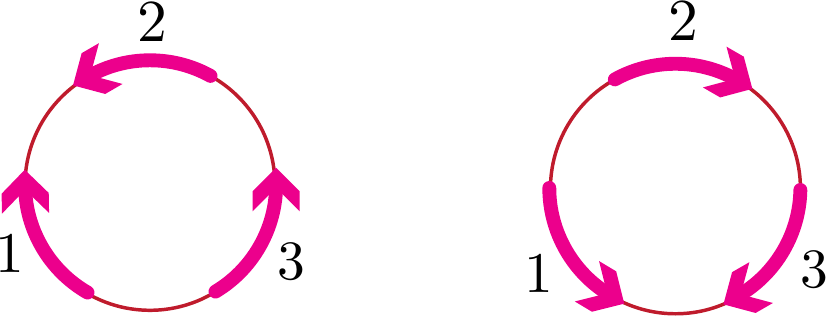}  } \\
 $\BG$ with $\BA=\{2,3\}.$  & &Steps P1  \& P2. & &  Step P3. \\ 
 \end{tabular}
 \end{center}
 \begin{center}
\begin{tabular}{ccc}
  \raisebox{5mm}{  \includegraphics[width=4cm]{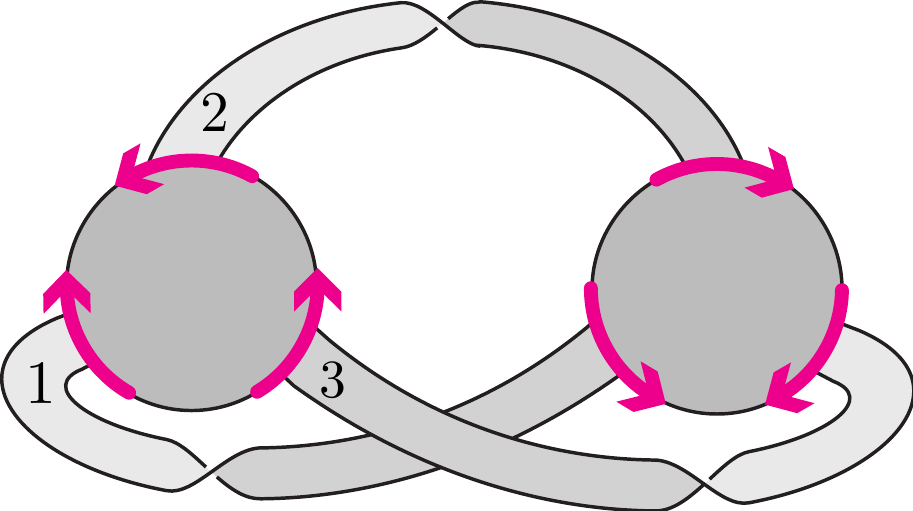}} & \hspace{3cm} & \includegraphics[width=4cm]{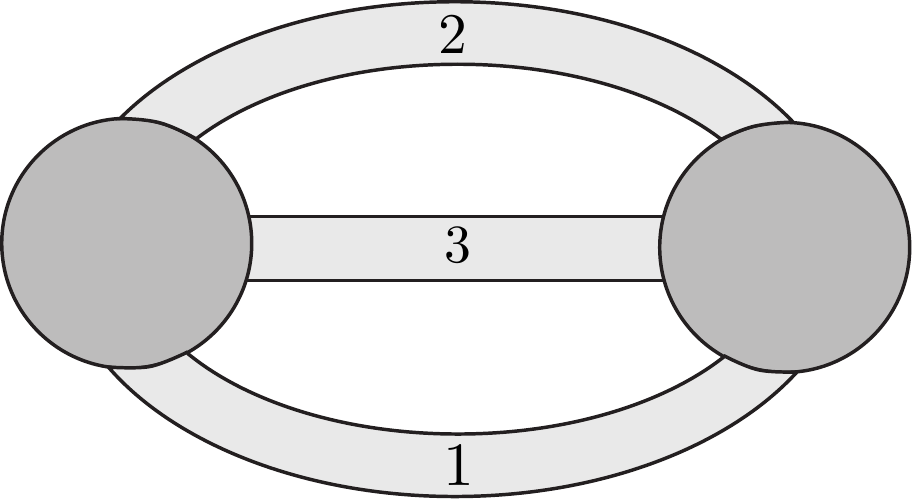}      \\
$\BG^{\BA} $& & Redrawing $\BG^{\BA}$.
\end{tabular}
\end{center}
\end{example}

\begin{example}\label{ex.dex2}
~

 \begin{center}
\begin{tabular}{ccc}
 \includegraphics[width=4cm]{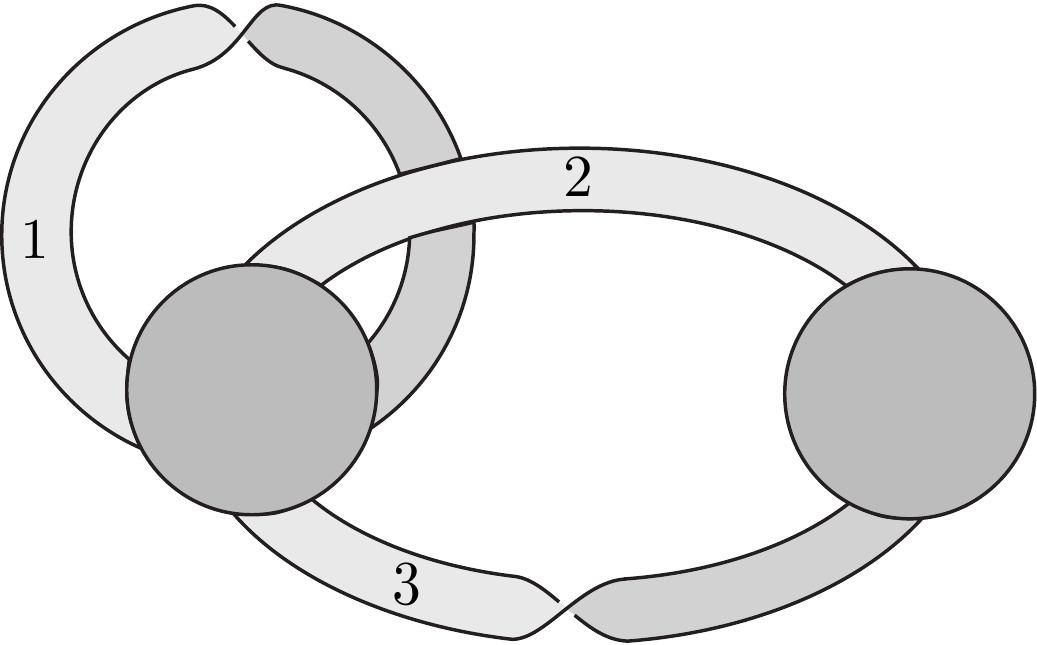}  &\hspace{3cm}& \includegraphics[width=4cm]{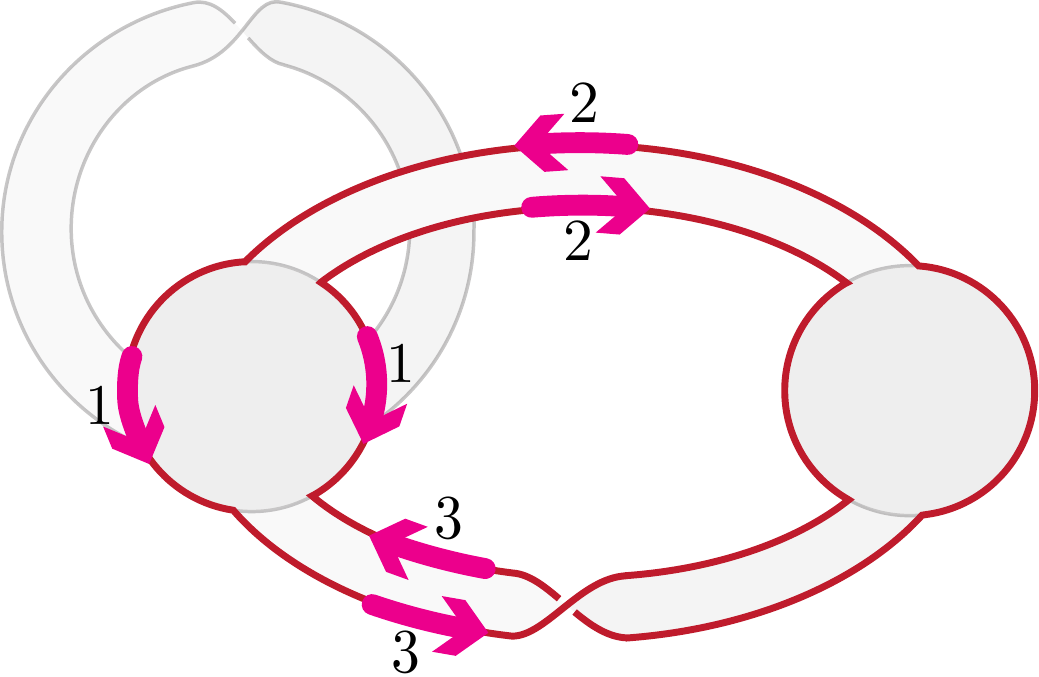}   \\
 $\BG$ with $\BA=\{2,3\}.$  & &Steps P1  \& P2.  \\ 
 \end{tabular}
 \end{center}
 \begin{center}
\begin{tabular}{ccc}
  \raisebox{5mm}{  \includegraphics[width=2cm]{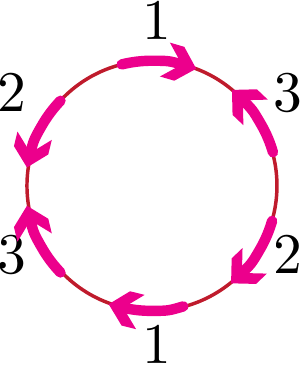}} & \hspace{3cm} & \includegraphics[width=4cm]{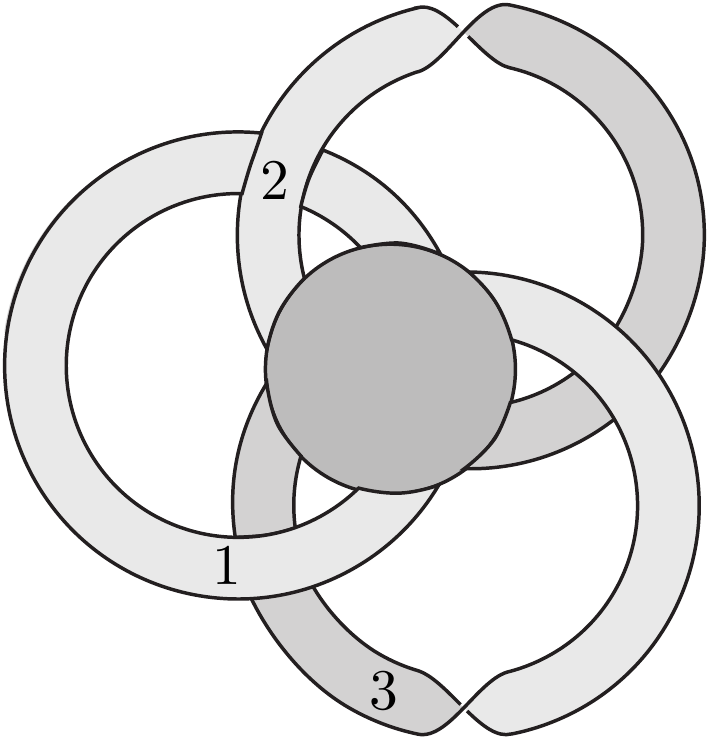}      \\
Step P3. & &  $\BG^{\BA} $.
\end{tabular}
\end{center}
\end{example}

Additional examples of partial duals can be found in \cite{Ch1} and \cite{Mo3}.

Notice that there is a correspondence between the edges of $\BG$ and $\BG^{\BA}$: every edge of $\BG$ gives rise to exactly two marking arrows of the same colour, and one edge of  $\BG^{\BA}$ is attached between these two arrows. We will denote the resulting natural bijection between the edge sets by
\[ \varphi: \E(\BG)\rightarrow \E(\BG^{\BA}). \]

\subsection{Natural duality}
Before continuing, we will record a few properties of partial duality. We are particularly interested in the connection between partial and natural duality.

\begin{definition}
Let $\BG=(\V(\BG), \E(\BG))$ be a ribbon graph.  We can regard $\BG$ as a punctured surface. By filling in the punctures using a set of discs denoted $ \V(\BG^*) $, we obtain a surface without boundary  $\Sigma$. 
The {\em natural dual} (or Euler-Poincar\'e dual)  of $\BG$  is the ribbon graph $\BG^* = (\V(\BG^*), \E(\BG)) $. \end{definition}

Note that the complementary sides of the edges are attached to the vertex set  in $\BG$ and $\BG^*$.  

For use later, we highlight the special case when $\BG=(\V(G), \emptyset)$. In this case $\BG^*=\BG$ and duality induces a natural bijection between $\V(\BG)$ and $\V(\BG^*)$.  

 We will often use shorthand notation and write   $\BG^*=\Sigma \backslash\V(\BG)$ when we mean that $\BG^*$ is the dual ribbon graph obtained through the surface $\Sigma$ as above.

The surface $\Sigma$ that arises by filling in the punctures of $\BG$ in the definition of natural duality  will be useful later. To record this concept, we define a {\em dual embedding}    $\{\BG, \BF, \Sigma\}$ of $\BG$ and $\BF$ into a surface $\Sigma$  to be an embedding of $\BG$  in a surface without boundary $\Sigma$  which has  the property that $\BF=\Sigma\backslash \V(\BG)$.

Note that  a dual embedding is independent of the order of the ribbon graphs $\BG$ and $\BF$ ({\em i.e.} the dual embeddings  $\{\BG, \BF, \Sigma\}$ and $\{\BF, \BG, \Sigma\}$ are equivalent).  Also note that ribbon graphs $\BG$ and $\BF$ are natural duals if and only if there exists a dual embedding $\{\BG, \BF, \Sigma\}$.

\bigskip

We can now describe a few properties of partial duality.
 \begin{lemma}\label{l.props}
Let $\BG$ be a ribbon graph, $\BA \subseteq \E(\BG)$ and $\BA^c=\E(\BG)\backslash \BA$.
Then 
\begin{enumerate}
\item  $\BG^{\E(\BG)} = \BG^*$;
\item $\left( \BG^{\BA}\right)^{\varphi (\BA)}=\BG$;
\item $  \BG^{\BA}\backslash  \varphi (\BA^c) =   \left( \BG\backslash \BA^c\right)^*$.
\end{enumerate}
\end{lemma}
\begin{proof}
Properties (1) and (2) are from \cite{Ch1}.

If  $\be \in\BA^c$, then the cycles defining the vertices of   $\BG^{\BA}$ follow the vertices incident with $\be$ in $\BG$ (see Figure~\ref{fig.pdd}). It then follows that we can delete the edges in $\BA^c$ before or after forming the partial dual and end up with the same ribbon graph. Thus  $\BG^{\BA}\backslash  \varphi (\BA^c)  = \left( \BG\backslash \BA^c \right)^{\BA}$.

It remains to show that $ \left( \BG\backslash \BA^c \right)^{\BA}=  \left( \BG\backslash \BA^c\right)^* $. But this is a consequence of Property (1) of the lemma as $\E(\BG\backslash \BA^c)= \BA$.
\end{proof}

\subsection{Partial dual embeddings}

\begin{lemma}\label{l.pd}
Let $\BG$ be a ribbon graph, $\BA\subset \E (\BG)$ and $\BA^c=\E(\BA)\backslash \BA$. Then the following construction gives $\BG^{\BA}$:
\begin{enumerate}\renewcommand{\theenumi}{Step P\arabic{enumi}${}^{\prime}$.}
\item Present $\BG$ as the arrow-marked ribbon graph $\ar{\BG\backslash\BA^c}$.
\item Take the natural dual of $\BG\backslash\BA^c$. The marking arrows on $\ar{\BG\backslash\BA^c}$ induce marking arrows on $\left(\BG\backslash\BA^c\right)^*$.
\item $\BG^{\BA}$ is the ribbon graph corresponding to the arrow-marked ribbon graph $\ar{(\BG\backslash\BA^c)^*}$.
\end{enumerate}
\end{lemma}
Before proving the lemma, we  provide an example of the construction.
\begin{example}\label{ex.dex3}
~

Reconsidering Example~\ref{ex.dex1} and carrying out partial duality using the recipe in the lemma gives:
 \begin{center}
\begin{tabular}{ccc}
 \includegraphics[width=4cm]{dex1}  & \hspace{5mm}\raisebox{10mm}{=}\hspace{5mm} & \includegraphics[width=4cm]{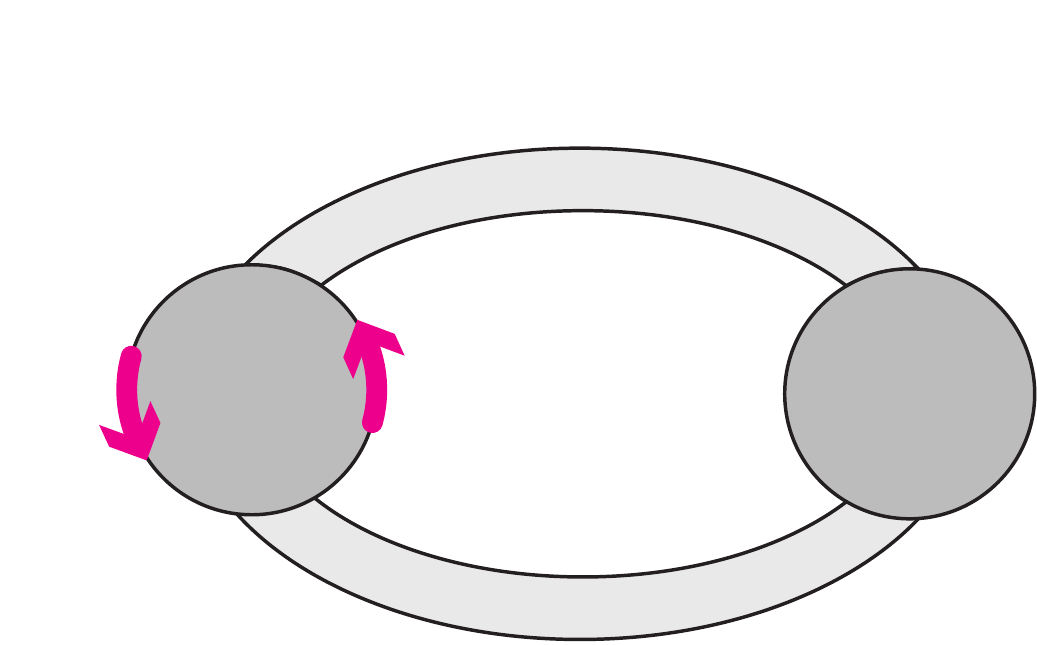}   \\ & & \\
 $\BG$ with $\BA=\{2,3\}.$  & &Step P1${}^\p$.  \\ 
 \end{tabular}
 \end{center}
 \begin{center}
\begin{tabular}{ccc}
  \includegraphics[height=4cm]{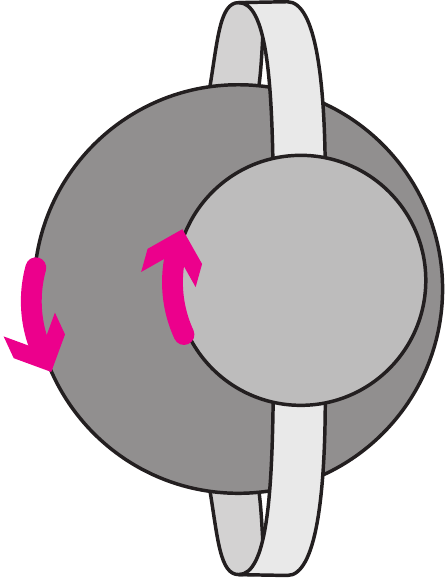}&\hspace{5mm}\raisebox{20mm}{=}\hspace{5mm}& \includegraphics[height=4cm]{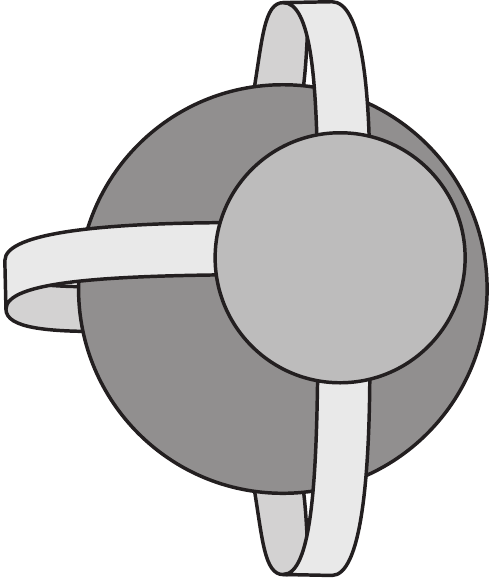}    \hspace{5mm}\raisebox{20mm}{=}\hspace{5mm} \raisebox{6mm}{\includegraphics[height=2.8cm]{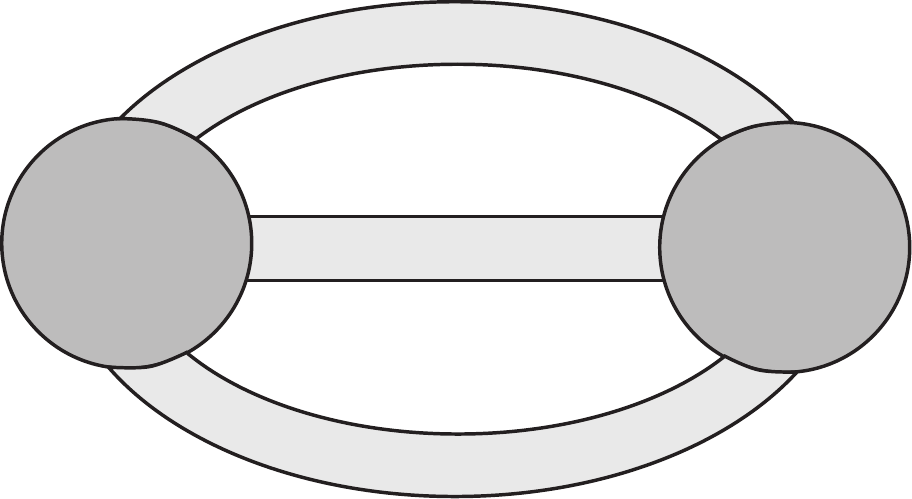}  }    \\ & & \\
Step P2${}^\p$. &&  $\BG^{\BA} $.
\end{tabular}
\end{center}
\end{example}

We now give a proof of Lemma~\ref{l.pd}.
\begin{proof}
Since $\BG^{\BA}$ and $(\BG\backslash \BA^c)^{\BA}$ have the same vertex set, the formation of $\BG^{\BA}$ from $\BG$ admits the following description: present $\BG$ as an arrow-marked ribbon graph $\ar{\BG\backslash \BA^c}$; form $\ar{(\BG\backslash \BA^c)^{\BA}}$, retaining the marking arrows from the last step; this arrow-marked ribbon graph describes $\BG^{\BA}$. 

Recalling from  Lemma~\ref{l.props}  that $  \BG^{\BA}\backslash  \varphi (\BA^c) =   \left( \BG\backslash \BA^c\right)^*$ and rewriting the second step of the above description using this fact gives the description of partial duality in the lemma.
\end{proof}

\bigskip

A key observation used in this paper is that by using the description of partial duality from Lemma~\ref{l.pd}, $\BG$ and $\BG^{\BA}$ can be described as a pair of naturally dual arrow-marked ribbon graphs.

\begin{definition}
A set  $\{\BG, \BF, \Sigma, \M\}$ is  a {\em partial dual embedding} of ribbon graphs $\BG$ and $\BF$ if
\begin{enumerate}\renewcommand{\labelenumi}{(\roman{enumi})}
\item $\{\BG, \BF, \Sigma\}$ is a dual embedding;
\item $\M$ is a set of disjoint coloured arrows marked on the boundaries of the embedded vertices in $\V(\BG) \cap \V(\BF) \subset \Sigma$ with the property that there are exactly two arrows of each colour.
\end{enumerate}
\end{definition}

\begin{theorem}\label{t.pdemrib}
Let $\BG$ and $\BF$ be ribbon graphs. Then $\BG$ and $\BF$ are partial duals if and only if there exists a partial dual embedding     $\{\widetilde{\BG}, \widetilde{\BF}, \Sigma, \M\}$ with the property that $\Sigma \backslash \V(\widetilde{\BF})\cup\M$  is an arrow-marked ribbon graph describing $\BG$, and 
$\Sigma \backslash \V(\widetilde{\BG})\cup\M$  is an arrow-marked ribbon graph describing $\BF$.
\end{theorem}
\begin{proof}
First suppose that $\BG$ and $\BF$ are partial duals. Then there exists a set of edges $\BA\subseteq \E(\BG)$ such that $\BG^{
\BA}=\BF$. Then $\BG$ can described as an arrow-marked ribbon graph $\ar{\BG\backslash \BA^c}$, where $\BA^c=\E(\BA)\backslash \BA$. Let $\Sigma$ be the surface obtained from $\BG\backslash \BA^c$ by filling in the punctures. Then 
$\{ \BG\backslash \BA^c, (\BG\backslash \BA^c)^*, \Sigma\}$ forms a natural dual embedding. The arrow markings on $\ar{\BG\backslash \BA^c}$ induce a set of coloured  arrows on $\V( \BG\backslash \BA^c) \cap \V(( \BG\backslash \BA^c)^*)$ with the property that there are exactly two arrows of each colour. Denote this induced set of coloured arrows by $\M$. Then 
\[ \{  \BG\backslash \BA^c, (\BG\backslash \BA^c)^*, \Sigma, \M  \} \]
is a partial dual embedding.
Moreover, $\Sigma \backslash \V( (\BG\backslash \BA^c)^*)$ describes $\BG$ by construction, and 
$\Sigma \backslash \V( (\BG\backslash \BA^c))$ clearly  describes $\BG^{\BA}=\BF$ if we use the construction of partial duality from Lemma~\ref{l.pd}.

\smallskip

Conversely, suppose that  $\{\widetilde{\BG}, \widetilde{\BF}, \Sigma, \M\}$ is a partial dual embedding with the property that $\Sigma \backslash \V(\widetilde{\BF})\cup\M$  is an arrow-marked ribbon graph describing $\BG$, and 
$\Sigma \backslash \V(\widetilde{\BG})\cup\M$  is an arrow-marked ribbon graph describing $\BF$. Then $\widetilde{\BG}$ and $ \widetilde{\BF}$ are precisely the naturally dual marked ribbon graphs described in Step~P2${}^{\prime}$ of the construction of the partial dual. Here $\BA$ is the set of edges of $\BG$ that are also in $\widetilde{\BG}$.

\end{proof}

The question of how graph theoretical properties of  partial duals $\BG^{\BA}$ relate to the original ribbon graph $\BG$ is of interest. As an application of Theorem~\ref{t.pdemrib}, we relate some graph theoretical properties of $\BG$ and $\BG^{\BA}$. For the application, we let $v(\BG):= |\V(\BG)|$, $e(\BG):=|\E(\BG)|$, $k(\BG)$ denote the number of connected components of $\BG$,   $  p(\BG)$ denote the number of boundary components of $\BG$, and, $g(\BG)$ denote the genus of $\BG$ when $\BG$ is orientable.
\begin{corollary}
Let $\BG$ be a ribbon graph and $\BA\subseteq \E(\BG)$.  Then 
\begin{enumerate}
\item $v(\BG^{\BA})=p(\BG \backslash\BA^c)$, where $\BA^c= \E(\BG) \backslash\BA$;
\item $p(\BG^{\BA}) = p(\BG \backslash\BA)$;
\item when  $\BG$ is orientable, $g(\BG^{\BA})= \frac{1}{2}\left(  2k(\BG)+e(\BG)- p(\BG \backslash\BA^c)- p(\BG \backslash\BA) \right)$.
\end{enumerate}
\end{corollary}
\begin{proof}
For the first identity, observe that since $\BG$ and $\BG^{\BA}$ are partial duals, there exists a partial dual embedding $ \{  \BG\backslash \BA^c, (\BG\backslash \BA^c)^*, \Sigma, \M  \}$ where $\BG$ and $\BG^{\BA}$  can be obtained from the partial dual embedding in the way described in Theorem~\ref{t.pdemrib}. Since process of obtaining ribbon graphs from partial dual embeddings does not create or destroy vertices, we have
\[   v(\BG^{\BA}) = v( (\BG\backslash \BA^c)^* ) = p(\BG\backslash \BA^c) =p(\BG \backslash\BA^c)  .\] 

For the second identity, consider the ribbon graph $\BG$ locally at an edge $\be$ as shown in the left hand figure below.
\begin{center}
\begin{tabular}{ccccc}
 \includegraphics[width=3cm]{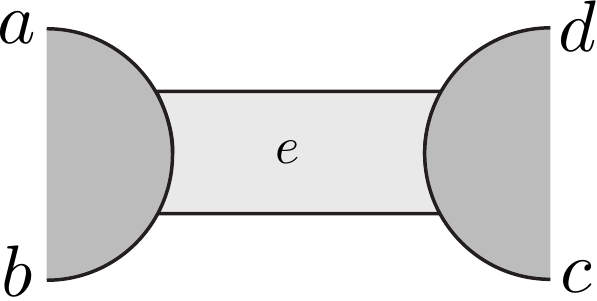} & \hspace{1cm} & \includegraphics[width=3cm]{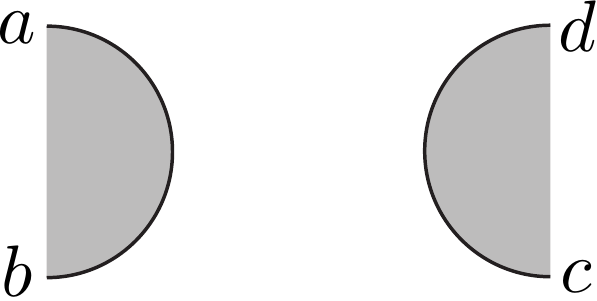} & \hspace{1cm} &   \includegraphics[width=3cm]{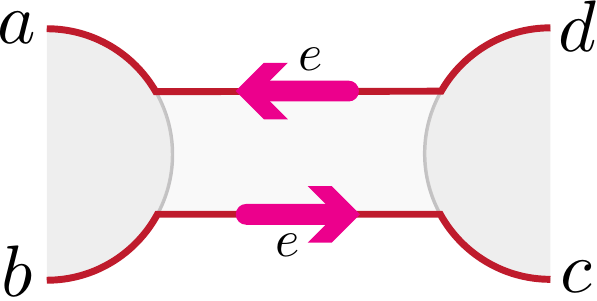}   \\
 An edge $\be$ in $\BG$. &  &$\BG\backslash \be$. & & Dual at $\be$. 
 \end{tabular}
\end{center}

 Let $a$, $b$, $c$ and $d$ be the points on the boundary components of $\BG$ as shown in the figure.
 To prove the result, it is enough to show that the arcs on the boundary components of $\BG^{\BA}$ and of $\BG \backslash\BA$ in the neighbourhood shown in the figure  connect $a$, $b$, $c$ and $d$  in the same way. There are two cases: when $\be \notin \BA$  and when $\be \in \BA$. If $\be \notin \BA$, then  $\BG^{\BA}$ and $\BG \backslash\BA$ are identical in a neighbourhood of $\be$ and the boundary components contain arcs $(a,d)$ and $(b,c)$ in both ribbon graphs. If $\be \in \BA$, then in  $\BG\backslash \be$ the boundary components contain arcs $(a,b)$ and $(c,d)$ (see the middle figure above), and in  $\BG^{\BA}$ the boundary components also contain arcs $(a,b)$ and $(c,d)$ (see the right hand figure above).

For the third identity, first note that by \cite{Ch1}, $\BG$ is orientable if and only if  $\BG^{\BA}$ is orientable. The result then follows from Euler's formula (which gives
$2g(\BG^{\BA})= 2k(\BG^{\BA})+e(\BG^{\BA})-v(\BG^{\BA})-p(\BG^{\BA})$), the first two identities of the theorem, and the facts that $e(\BG)=e(\BG^{\BA}) $ and $k(\BG)=k(\BG^{\BA}) $ from \cite{Ch1}.

\end{proof}

\section{Partial duality for graphs}\label{s.pdg}

We will always denote ribbon graphs, their edges and their vertices using a bold font and reserve the non-bold font for graphs. Just as with ribbon graphs, if $G$ is a graph we let $\E(G)$ denote its edge set  and $\V(G)$ denote its vertex set. 

\subsection{Natural duality}\label{ss.nd}

If $\BG=(\V(\BG),\E(\BG))$ is a ribbon graph then we can construct a graph $G=(\V(G),\E(G))$ from $\BG$ by replacing each edge of $\BG$ with a line, and then contracting the vertices of $\BG$ into points. Such a graph $G$ is called the {\em core} of $\BG$. 

Notice that there is a natural correspondence  between the edges  of a ribbon graph and its core, and the vertices of a ribbon graph and its core.

Recall that 
two graphs  are said to be (Euler-Poincar\'e) {\em dual graphs} if and only if the graphs  are the cores of naturally dual ribbon graphs. As usual, the dual of $G$ is denoted by $G^*$. Duality acts disjointly on connected components.

There is also a  canonical embedding of the core $G$ of a ribbon graph $\BG$ into $\BG$: place each vertex of $G$ inside the corresponding vertex of $\BG$; and place each edge of $G$ along the corresponding edge of $\BG$. 
Also, if $G$ is a graph embedded in a surface, then we can form a ribbon graph $\BG$ by taking  a small neighbourhood  of $G$.  $G$ is then the core of $\BG$.

\bigskip

We say that $\{ G, G^*, \Sigma\}$ is a {\em dual embedding of graphs} if there are neighbourhoods of $G$ and $G^*$ in $\Sigma$ defining ribbon graphs $\BG$ and $\BG^*$ such that $\{\BG,\BG^*,\Sigma\}$ is a dual embedding of ribbon graphs.

A little care needs to be taken with vertices that meet no edges, call these {\em isolated vertices}. The dual of an isolated vertex is an isolated vertex, and a dual embedding of an isolated vertex and its dual consists of an embedding of these two vertices in a sphere. In particular this means that if $G$ contains $k$ isolated vertices $v_1, \ldots , v_k$, then the dual embedding  $\{ G, G^*, \Sigma\}$ will contain $k$ spherical components $S^2_1, \ldots , S^2_k$, such that $S^2_i$ contains the vertex $v_i$ and exactly one isolated vertex of $G^*$, for $i=1, \ldots k$. Notice that this means that a dual embedding induces  a pairing of the isolated vertices of $G$ and $G^*$.

Dual graphs can be characterized in terms of a bijection between edge sets. To describe this characterization we need to introduce some terminology.

Suppose that $G$ and $H$ are graphs and $\varphi:\E(G)\rightarrow \E(H)$ is a bijection between their edge sets. Let $v\in \V(G)$  and $S_v$ be  the set of edges that are incident with $v$. Then the set   $\varphi(S_v)$  of edges in $H$ together with the vertices that are incident with $\varphi(S_v)$ form a subgraph of $H$. This subgraph is denoted $H_v$. 

\begin{definition}\label{d.ed}
Let $G$ and $H$ be graphs and $\varphi:\E(G)\rightarrow \E(H)$ be  a bijection. We say that $\varphi$ satisfies {\em Edmonds' Criteria} if 
\begin{enumerate}\renewcommand{\labelenumi}{(\roman{enumi})}
\item  edges $e,f\in \E(G)$ belong to the same connected component if and only if $\varphi(e), \varphi(f)\in \E(H)$ belong to the same connected component;
\item for each $v\in \V(G)$, $H_v$ is Eulerian;
\item for each $v\in \V(H)$, $G_v$ is Eulerian.
\end{enumerate}
\end{definition}

The following lemma is easily verified. It is essentially one implication of Edmonds' Theorem (a formal proof of the lemma can  easily be deduced from  \cite{Ed}).
\begin{lemma}\label{l.ed}
 Let $\{G,H,\Sigma\}$ be a dual embedding of the graphs $G$ and $H$. Define a bijection $\varphi: \E(G)\rightarrow \E(H)$ by setting $\varphi(e)$ to be the unique edge of $H$ that $e$ intersects.  Then $\varphi$ satisfies Edmond's Criteria.
\end{lemma}

Edmonds' Criteria provides a characterization of natural duality.
\begin{theorem}[Edmonds \cite{Ed}]
Two graphs $G$ and $H$ are natural duals if and only if there exists a bijection $\varphi:\E(G)\rightarrow \E(H)$ that satisfies Edmonds' Criteria and the number of vertices in $G$ and $H$ that are incident to no edges is equal.
\end{theorem}

We will now give a brief overview of the idea of the proof of Edmonds' Theorem, referring the reader to \cite{Ed} for details. On one hand, if $G$ and $H$ are dual then there is a dual embedding $\{ G,H,\Sigma \}$ and we can use Lemma~\ref{l.ed} above to find a suitable bijection. 
Conversely, suppose  there is an edge bijection satisfying the conditions in the theorem. Consider $H_v$. By condition (ii) of Edmonds' Criteria, $H_v$ contains an Eulerian circuit. Therefore it can be obtained by identifying edges of a cycle. We can assume this cycle bounds a polygon and that $v$ and its incident half-edges are embedded  in this polygon in such a way that an end of every half-edge lies on exactly one side of the polygon. 
And there is exactly one edge meeting each side of the polygon. 
Glue together the sides of  all of the polygons which arise in this way in such a way  that the embedded half-edges that come from the same edge of $G$ are identified. Condition (iii) of Edmonds' Criteria ensures this results in a surface. This gives a dual embedding $\{G,H,\Sigma\}$, so the graphs are dual as required. 

We will need one corollary of Edmonds' theorem.  The corollary is in fact a step from Edmonds' proof of his 
theorem in \cite{Ed}. As the corollary  follows immediately from this reference,  we will omit its proof. 

\begin{corollary}\label{c.ed2}
Let $\varphi: \E(G)\rightarrow \E(H)$ be a bijection that satisfies Edmonds' Criteria and $\{G,H,\Sigma\}$ be a dual embedding constructed by Edmonds' Theorem. Then for each  $v\in \V(G)$,  if we  cut the surface $\Sigma$ along the subgraph $H_v$, the component that contains $v$ is a   surface that has been obtained by identifying some of the vertices of a polygon. 
\end{corollary}

\subsection{Partial duality of graphs}\label{ss.pdg}
\begin{definition}
We say that two graphs are {\em partial duals} if they are cores of partially dual ribbon graphs.
\end{definition}

Let $G$ be a graph and $A\subseteq \E(G)$. By the notation $G^A$ we mean that $G^A$ is the core of $\BG^{\BA}$ where $G$ is the core of $\BG$ and $\BA$ is the edge set of $\BG$ that corresponds with $A$. 

We have seen that partially dual ribbon graphs can be characterized by  the existence of an appropriate partially dual embedding. A corresponding result holds for partial dual graphs. To describe the corresponding result, we make the following definition:
\begin{definition}
A {\em partial dual embedding} of graphs is a set
\[  \{ \widetilde{G}, \widetilde{H}, \Sigma, E   \}  \]  
where $\Sigma$ is an  surface without boundary,  $\widetilde{G}, \widetilde{H} \subset \Sigma$ are embedded graphs and $E$ is a set of coloured edges that are embedded in $\Sigma$ 
such that
\begin{enumerate}\renewcommand{\labelenumi}{(\roman{enumi})}
\item only the ends of each embedded edge in $E$ meet $\widetilde{G} \cup \widetilde{H} \subset\Sigma$;
\item $\{\widetilde{G},\widetilde{H},\Sigma\}$ is a dual embedding;
\item each edge in $E$ is incident to one vertex in $V(\widetilde{G})$ and one vertex in $V(\widetilde{H})$;
\item there are exactly two edges of each colour in $E$. 
\end{enumerate}
\end{definition}

\begin{theorem}\label{t.pdem}

Two graphs $G_1$ and $G_2$ are partial duals if and only if there exists a partial dual embedding $  \{ \widetilde{G}_1, \widetilde{G}_2,  \Sigma, E   \} $ such that for each $i$, 
$G_i$ is obtained from $\widetilde{G}_i$ by adding an edge  between the vertices of  $\widetilde{G}_i$ that are incident with the two edges  in $E$ that have the same colour, for each colour. 
\end{theorem}
\begin{example}\label{ex.pdem}
An example of a partial dual embedding is
\[\includegraphics[height=3cm]{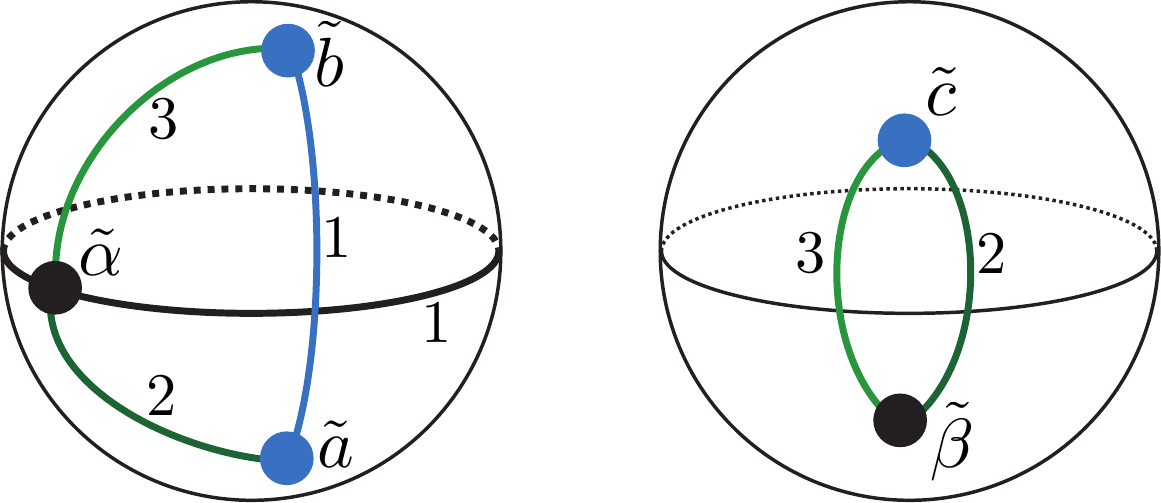},\]
where $\Sigma$ is the disjoint union of two spheres, $\widetilde{G}_1 =\left( \{ \tilde{\alpha}, \tilde{\beta} \}, \{1\} \right)$and  $\widetilde{G}_2=\left( \{ \tilde{a}, \tilde{b}, \tilde{c} \}, \{1\} \right)$. 
Following the recipe in the theorem we recover the graphs 
\[\raisebox{8mm}{$G=$ }\;\;\includegraphics[width=3cm]{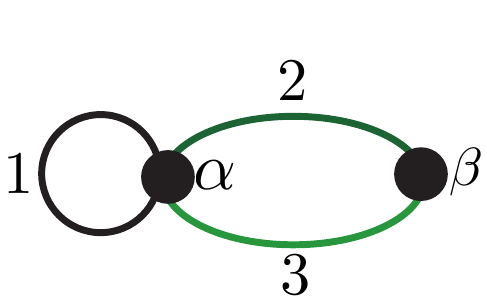} \quad\quad  \raisebox{8mm}{and}  \quad\quad  \raisebox{8mm}{$G^{\{1\}}=$ } \;\;\;\raisebox{-5mm}{\includegraphics[width=2.4cm]{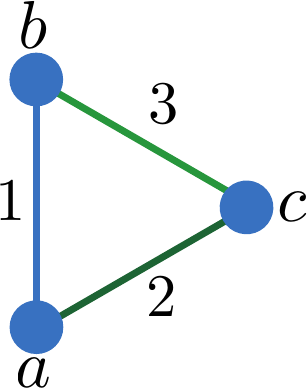}.}\]
These graphs are indeed partial duals as they are the cores of
\[  \raisebox{9mm}{$\BG=$ }   \includegraphics[width=4cm]{dex1}  \quad \quad  \raisebox{9mm}{and}  \quad\quad \raisebox{9mm}{$\BG^{\{1\}}=$ }\raisebox{-4mm}{\includegraphics[width=4cm]{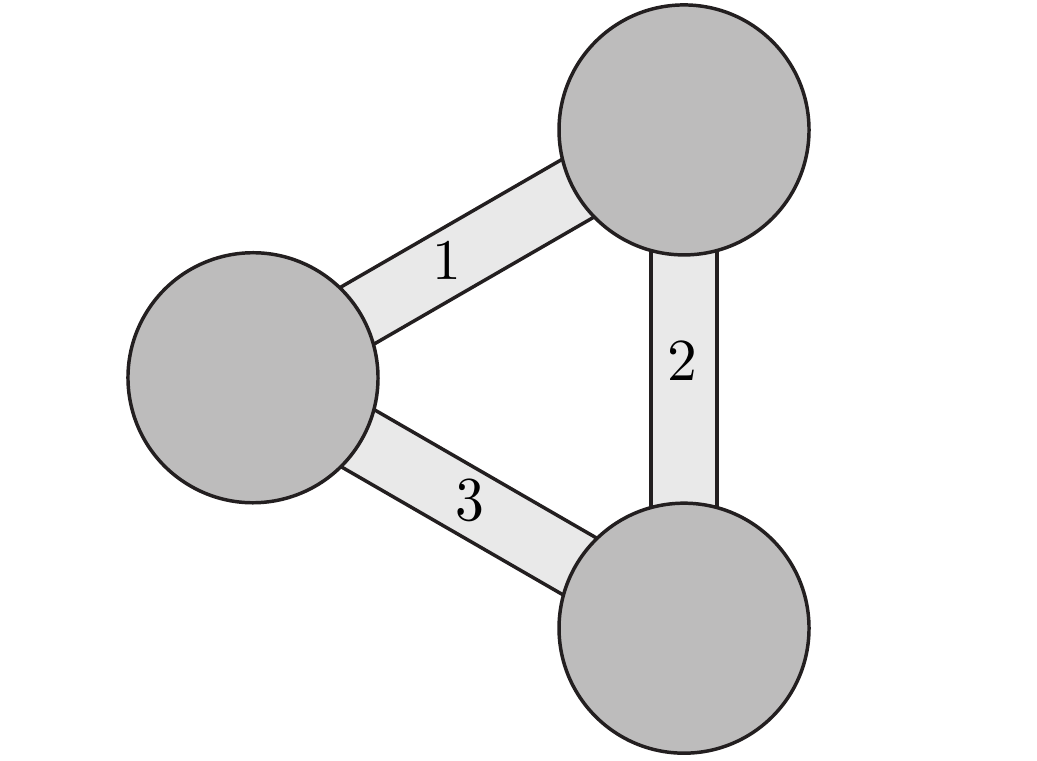}}\]
respectively.

\end{example}

We will now prove Theorem~\ref{t.pdem}. The idea behind the proof is to construct a correspondence between partial dual embeddings of ribbon graphs and their (embedded) cores. It then follows by Theorem~\ref{t.pdemrib} that the graphs constructed by the theorem are the cores of partially dual ribbon graphs.
\begin{proof}
First suppose that $G_1$ and $G_2$ are partial duals, so $G_1$ and $G_2$ are the cores of partially dual ribbon graphs. Then, by Theorem~\ref{t.pdemrib}, there exists 
a partial dual embedding     $\{\widetilde{\BG}_1, \widetilde{\BG}_2, \Sigma, \M\}$ such that $\Sigma \backslash \V(\widetilde{\BG}_2)\cup \M$  is an arrow-marked ribbon graph describing $\BG_1$; 
$\Sigma \backslash \V(\widetilde{\BG}_1)\cup \M$  is an arrow-marked ribbon graph describing $\BG_2$;
$G_1$ is the core of $\BG_1$; and $G_2$ is the core of $\BG_2$.

A partial dual embedding of graphs    $  \{ \widetilde{G}_1, \widetilde{G}_2,  \Sigma, E   \} $ can be constructed from $\{\widetilde{\BG}_1, \widetilde{\BG}_2, \Sigma, \M\}$ in the following way: let $\widetilde{G}_1$ be the canonically embedded core of $\widetilde{\BG}_1$ and  let $\widetilde{G}_2$ be the canonically embedded core of $\widetilde{\BG}_2$.  Each arrow on $\Sigma$ meets exactly two vertices of $\widetilde{\BG}_1\cup \widetilde{\BG}_2$. For each arrow, add an embedded edge between the two corresponding vertices of the graph $\widetilde{G}_1\cup \widetilde{G}_2\subset\Sigma$ which passes through this arrow. Colour the edge with the colour of the arrow that it passes through. The set of edges added in this way forms $E$. 

We need to show that $  \{ \widetilde{G}_1, \widetilde{G}_2,  \Sigma, E   \} $ is indeed a partial dual embedding of graphs and that the graphs $G_1$ and $G_2$ can be recovered from the partial dual embedding in the way described by the theorem.  

To see that $  \{ \widetilde{G}_1, \widetilde{G}_2,  \Sigma, E   \} $ is a partial dual embedding, first note that by construction  $\widetilde{G}_1$, $\widetilde{G}_2$ and $E$ are all embedded in $\Sigma$, and that only the ends of the edges in $E$ meet $\widetilde{G}_1$ or $\widetilde{G}_2$.    $  \{ \widetilde{G}_1, \widetilde{G}_2,  \Sigma  \} $ is a dual embedding since   $\{\widetilde{\BG}_1, \widetilde{\BG}_2, \Sigma\}$ is. Since each arrow in $\M$ meets one vertex in $\V(\BG_1)$ and one vertex in $\V(\BG_2)$, each edge in $E$ is incident to one vertex in $V(\widetilde{G}_1)$ and one vertex in $V(\widetilde{G}_2)$. The colouring requirement follows since there are exactly two edges of each colour in $\M$ and the edge colourings of $E$ are induced from $\M$. 

Finally, $\BG_i$ can be recovered from $\widetilde{\BG}_i\cup\M$ by adding edges between the marking arrows of the same colour.  Therefore, if $\bu$ and $\bv$ are vertices of $\widetilde{\BG}_i$ which are marked with an arrow of the same colour and $u$ and $v$ are the corresponding vertices of $\widetilde{G}_i$, then to construct the core of $\BG_i$ we need to add an edge between $u$ and $v$. But since $u$ and $v$ are each  incident with edges in $E$ of the same colour, we need to add an edge  between the vertices of  $\widetilde{G}_i$ that are incident with the two edges  in $E$ of the same colour.  This is exactly the construction described in the statement of the theorem. Doing this for each colour gives $G_i$, completing the proof of necessity.

\smallskip

Conversely, suppose that $  \{ \widetilde{G}_1, \widetilde{G}_2,  \Sigma, E   \} $ is a partial dual embedding and that $G_1$ and $G_2$ are obtained as described in the statement of the theorem.
Construct a partial dual embedding $\{\widetilde{\BG}_1, \widetilde{\BG}_2, \Sigma, \M\}$ of ribbon graphs in the following way: take a small neighbourhood in $\Sigma$ of the embedded graph $\widetilde{G}_1$ to form $\widetilde{\BG}_1$; let $\widetilde{\BG}_2 = ( \Sigma\backslash \widetilde{\BG}_1, \E (\widetilde{\BG}_1) )$; wherever an edge in $E$ meets a boundary of vertices add an arrow pointing in an arbitrary direction which is coloured by the colour of the edge in $E$.  $\M$ is the set of such coloured arrows.  

    To see that $\{\widetilde{\BG}_1, \widetilde{\BG}_2, \Sigma, \M\}$ is a partial dual embedding, note that 
$\{\widetilde{\BG}_1, \widetilde{\BG}_2, \Sigma\}$ is a dual embedding since  $  \{ \widetilde{G}_1, \widetilde{G}_2,  \Sigma   \} $ is, and that there are exactly two arrows of each colour since there are exactly two edges of each colour in $E$.  

Let $\BG_i$ denote the ribbon graph described by the arrow-marked ribbon graph $\widetilde{\BG}_i\cup\M$. Then  $G_i$ is the core of $\BG_i$ (since whenever an edge is added between two vertices of $\widetilde{G}_i$ in the formation of $G_i$, an edge is added between the corresponding vertices of  $\widetilde{\BG}_i$ in the formation of $\BG_i$). Finally, $G_1$ and $G_2$ are partial dual graphs  since, by 
Theorem~\ref{t.pdemrib}, $\BG_1$ and $\BG_2$ are partial dual ribbon graphs.
\end{proof}

The corollary below follows from the construction of a partial dual embedding in the proof above.
\begin{corollary}\label{c.pdem}
If $G$ and $G^A$ are partial duals then the corresponding partial dual embedding as constructed by Theorem~\ref{t.pdem} is $\{G\backslash A^c ,  G^A\backslash \varphi (A^c), \Sigma, E\}$, where $A^c= \E(G)\backslash A$. Moreover, 
$G$ (respectively $G^A$) is obtained from $G\backslash A^c$  (respectively $G^A\backslash \varphi (A^c)$)  by adding an edge  between the vertices of  $G\backslash A^c$  (respectively $G^A\backslash \varphi (A^c)$) that are incident with the two edges  in $E$ that have the same colour, for each colour. 

\end{corollary}

\begin{definition}
If $G$ and $H$ are partially dual graphs that can be obtained from a partial dual embedding $\{\widetilde{G}, \widetilde{H}, \Sigma, E\}$ in the way described by Theorem~\ref{t.pdem}, then we say  that $\{\widetilde{G}, \widetilde{H}, \Sigma, E\}$ is a {\em partial dual embedding} for $G$ and $H$.
\end{definition}

We will use the following basic observations about partial dual embeddings. Suppose $\{\widetilde{G}, \widetilde{H}, \Sigma, E\}$ is a partial dual embedding for $G$ and $H$.
Then since $G$ is obtained from $\widetilde{G}$ by adding edges, $\widetilde{G}$ is a spanning subgraph of $G$. Similarly $\widetilde{H}$ is a spanning subgraph of $H$. In particular, this means that there is a natural bijection between the vertices of $G$ and  of $\widetilde{G}$, and a natural bijection between  the vertices of $H$ and  of $\widetilde{H}$.  
The partial dual embedding induces a natural bijection between the edges of $G$ and the edges of $H$ in the following way. First observe that by the way $G$ is obtained from the partial dual embedding, the set of colours of the edges in $E$ together with the colours of the edges in $\E(\widetilde{G})$ is in one-to-one correspondence with the set of colours of the edges in $\E(G)$. This correspondence induces a bijection $\psi_1$ between the sets of colours. Similarly, there is a bijection $\psi_2$ between the set of colours of of the edges in $E$ together with the colours of the edges in $\E(\widetilde{H})$ to the set of colours of the edges in $\E(H)$. In addition to these two bijections, the dual embedding $\{\widetilde{G}, \widetilde{H}, \Sigma\}$ provides a natural bijection between $\E(\widetilde{G})$ and $\E(\widetilde{H})$ (and edge $e$ is mapped to the unique edge it intersects). This induces a natural bijection $\psi_3$ between $E$ together with the colours of the edges in $\E(\widetilde{G})$ and $E$ together with the colours of the edges in $\E(\widetilde{H})$. 
Then  $\psi_2\circ\psi_3\circ(\psi_1)^{-1}$  is the required natural bijection from  $\E(G)$ to $\E(H)$ induced by $\{\widetilde{G}, \widetilde{H}, \Sigma, E\}$. 
Notice that since $G$ and $H$ were obtained from a partial dual embedding, we have $\psi_2\circ\psi_3\circ(\psi_1)^{-1} = \varphi$. Also note that any colouring of one of $\E(G)$, $\E(H)$, $\E(\widetilde{G} )$ and  $\E(\widetilde{H} )$ induces a colouring on the other three.

\subsection{A generalization of Edmonds' Theorem}\label{ss.ged}

\begin{theorem}\label{t.main}
Two graphs $G$ and $H$ are partial duals if and only if there exists a bijection $\varphi: \E(G) \rightarrow \E(H)$, such that 
\begin{enumerate}
\item $\left. \varphi\right|_{A}:  A\rightarrow \varphi(A)$ satisfies Edmonds' Criteria for some subset $A\subseteq \E(G)$;
\item If $v\in \V(G)$ is incident to an edge in $A$, and if $e\in \E(G)$ is incident to $v$, then $\varphi(e)$ is incident to a vertex of $\varphi(A)_v$. Moreover, if both ends of $e$ are incident to $v$, then both ends of $\varphi(e)$ are incident to vertices of $\varphi(A)_v$.

\item If $v\in \V(G)$ is not incident to an edge in $A$, then there exists a vertex $v' \in \V(H)$ with the property that  $e\in \E(G)$ is incident to $v$ if and only if  $\varphi(e)\in \E(H)$ is incident to $v'$. Moreover, both ends of $e$ are incident to $v$ if and only if both ends of $\varphi(e)$ are incident to $v'$.

\end{enumerate}
Here $\varphi(A)_v $ is the subgraph of $H$ induced by the images of the edges from $A$ that are incident with $v$. 
\end{theorem}

In the statement of the theorem it is to be understood that when we say  $\left. \varphi\right|_{A}$ satisfies Edmonds' Criteria the graphs $H_v$ and $G_v$ of Definition~\ref{d.ed} are subgraphs of $ G\backslash A^c$ and $H\backslash \varphi(A^c)$ respectively.

For ease of comprehension we will separate the proof of Theorem~\ref{t.main} in to the case of sufficiency and necessity. We prove sufficiency first.  

\bigskip

The idea behind the proof of sufficiency is to use the bijection $\varphi$ to construct a partial dual embedding for $G$ and $H$. In a little more detail, suppose we are given the graphs 
$G$ and  $H$, a bijection $\varphi: \E(G) \rightarrow \E(H)$ and a subset $A\subseteq \E(G)$ that satisfy the conditions in the theorem. Then since $\left. \varphi\right|_{A}$ satisfies 
Edmonds' Criteria, we can construct a dual embedding $\{ \widetilde{G},  \widetilde{H},  \Sigma\}$, with   $\widetilde{G} = G\backslash A^c$ and  $\widetilde{H}=  H \backslash \varphi(A^c) $. (A little care has to be taken in cases where $G\backslash A^c$ contains vertices which are not incident to any edges, if $v$ is such a vertex and $v'$ is the vertex of $H$ guaranteed by the third condition in the theorem, then choose a  dual embedding $\{ \widetilde{G},  \widetilde{H},  \Sigma\}$ in which $v$ and $v'$ are embedded in the same component of $\Sigma$ for each such $v$.)  
We can obtain a set of partial dual embeddings from the dual embedding $\{ \widetilde{G},  \widetilde{H},  \Sigma\}$ as follows: for each end of each edge $e\in A^c$ that is incident to a vertex $v$, embed an $e$-coloured edge between the vertex $v$ of $\widetilde{G}$ and one of the vertices of $\widetilde{H}$ that bounds the face that $v$ lies in.  
As we will see, the second and third conditions on $\varphi$ in the statement of the theorem ensures that a partial dual embedding for $G$ and $H$ is among the partial dual embeddings that can be constructed in this way (we will explicitly construct a partial dual embedding for $G$ and $H$ for $G$ and $H$ in the proof of the theorem.)

\begin{example}
The construction of a partial dual embedding for $G$ and $H$ from $\varphi$ is illustrated in this example. The bijection $\varphi$ is defined formally in the proof of sufficiency.
Let 
\[  G=  \raisebox{-8mm}{\includegraphics[height=16mm]{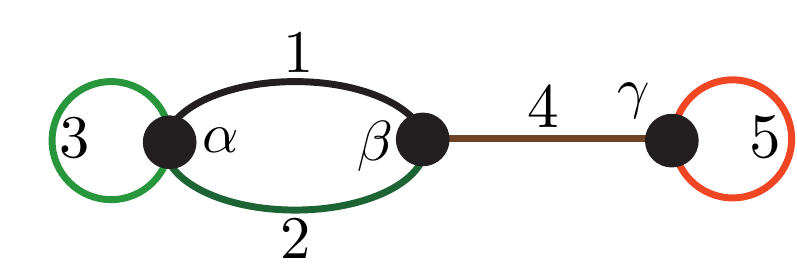}} \quad \quad \quad \text{and }\quad \quad \quad 
H=  \raisebox{-8mm}{\includegraphics[height=16mm]{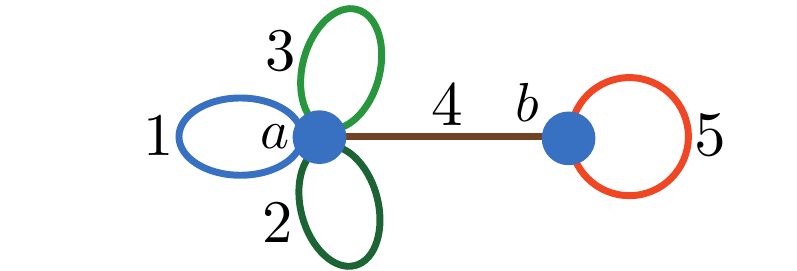}}.
\]
Let $\varphi: \E(G)\rightarrow\E(H)$ by $\varphi(i)=i$, for $i=1, \ldots, 5$. It is easily checked that $\left. \varphi\right|_{A}$ satisfies 
Edmonds' Criteria with $A=\{1\}$ and $\gamma'=b$. To construct the partial dual embedding, first use the fact that  $\left. \varphi\right|_{A}$ satisfies 
Edmonds' Criteria to construct a dual embedding:
\[  \includegraphics[height=30mm]{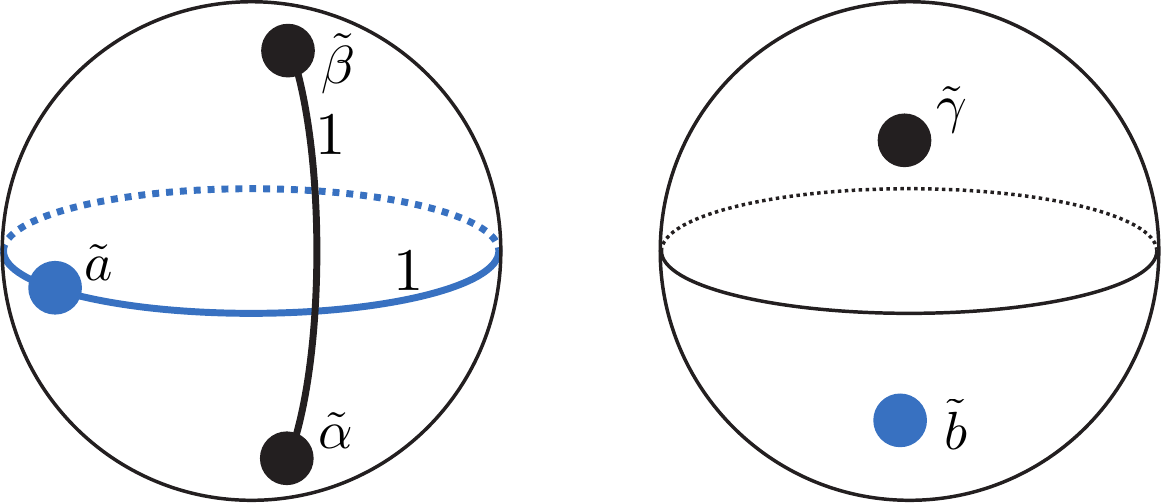}. \]
It remains to add the set $E$ of coloured embedded edges. We do this colour by colour. 

$2$-coloured edges: edge $2$ is incident to two vertices, $\alpha$ and $\beta$, that are incident to  an edge of $A$. $\varphi(2)$ is incident to $a$.  $a\in \varphi(A)_{\alpha}$ and  $a\in \varphi(A)_{\beta}$ so we can  embed a $2$-coloured edge between  $\tilde{a}$ and $\tilde{\alpha}$ and another $2$-coloured edge between  $\tilde{a}$ and $\tilde{\beta}$.

$3$-coloured edges: both ends of edge $3$ are incident to  $\alpha$, $\alpha$ is incident to an edge of $A$, and $\varphi(3)$ is incident to $a$. As $a\in \varphi(A)_{\alpha}$,  we can  embed two $3$-coloured edges between  $\tilde{a}$ and $\tilde{\alpha}$.

$4$-coloured edges: edge $4$ is incident to $\beta$ and $\gamma$. Only $\beta$ is incident to a vertex of $A$. $\varphi(4)$ is incident to $a$ and $b$. $b=\gamma'$ so embed a $4$-coloured edge between $\tilde{b}$ and $\tilde{\gamma}$. $a\in \varphi(A)_{\beta}$ so embed a $4$-coloured edge between $\tilde{a}$ and $\tilde{\beta}$. 

$5$-coloured edges: both ends of edge $5$ are incident to   $\gamma$. $\gamma$ does not meet an edge of $A$. $\varphi(5)$ is incident to $b$ and  $b=\gamma'$ so embed two  $5$-coloured edges between $\tilde{b}$ and $\tilde{\gamma}$.

This gives a (non-unique) partial dual embedding 
\[  \includegraphics[height=30mm]{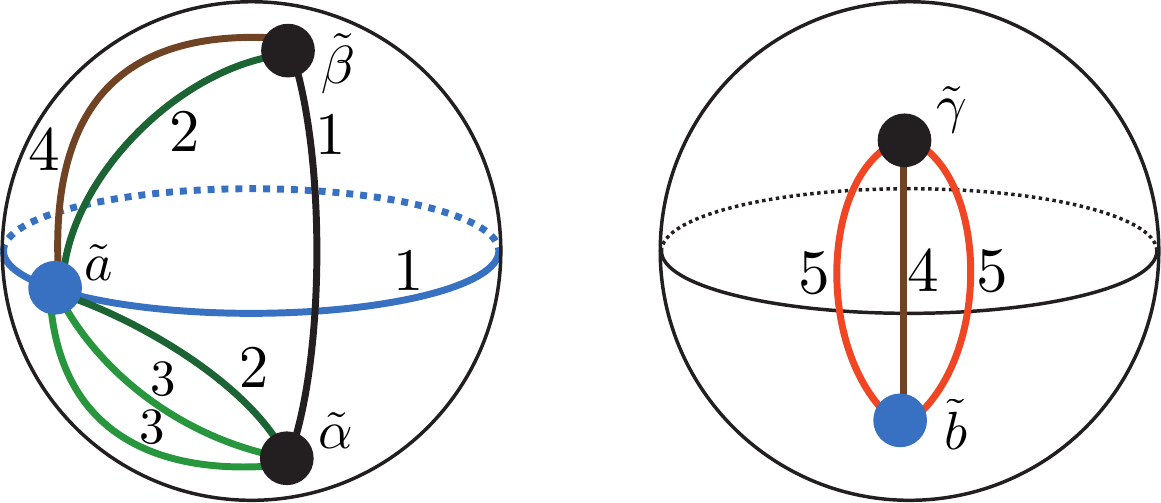}. \]
This partial dual embedding is necessarily  a partial dual embedding for $G$ and $H$.

\end{example}

\begin{proof}[Proof of Theorem~\ref{t.main} (Sufficiency).]
Given graphs  $G$ and  $H$, assume  $\varphi: \E(G) \rightarrow \E(H)$ and $A\subseteq \E(G)$ satisfy the conditions given in the statement of Theorem~\ref{t.main}.
Since $\left. \varphi\right|_{A}$ satisfies 
Edmonds' Criteria, we can construct a dual embedding $\{ \widetilde{G},  \widetilde{H},  \Sigma\}$, where   $\widetilde{G} = G\backslash A^c$,  $\widetilde{H}=  H \backslash \varphi(A^c) $ and $A^c:=\E(G)\backslash A$.  
  $\widetilde{G}$ may contain vertices that are not incident to an edge in $A$. For each vertex $v\in \V(G)$ that is not incident with any edge in $A$ and each associated vertex $v' \in \V(H)$ given by the third condition in the theorem,  choose a  dual embedding $\{ \widetilde{G},  \widetilde{H},  \Sigma\}$ in which the vertices of $\tilde{v} \in \V(\widetilde{G})$ and $\tilde{v}' \in \V(\widetilde{H})$, corresponding to $v$ and $v'$, are embedded in the same (spherical) component of $\Sigma$.

We will now construct a partial dual embedding $\{ \widetilde{G},  \widetilde{H},  \Sigma, E\}$ from this dual embedding.  Let $e\in A^c$ be an edge of $G$ and let $e$ be incident to (not necessarily distinct) vertices $u$ and $v$ in $\V(G)$. Further suppose that the edge $\varphi(e)$ of $H$ is incident to (not necessarily distinct) vertices $a$ and $b$ in $\V(H)$.
Let $\tilde{u}$ and $\tilde{v}$ respectively  denote the vertices in $\widetilde{G}$ corresponding  to the vertices $u$ and $v$ of $G$; and let $\tilde{a}$ and $\tilde{b}$ respectively  denote the vertices in $\widetilde{H}$ corresponding  to the vertices $a$ and $b$ of $H$. 
There are three cases to consider: if $u$ and $v$ are both incident to an edge of $A$,  if exactly one of $u$ or $v$ is incident to an edge of $A$, and if neither $u$ and $v$ is incident to an edge of $A$.

 First suppose both  $u$ and $v$ are incident to an edge of $A$. Then by the second condition of the theorem,  either $a\in \varphi(A)_u$ or  $a\in \varphi(A)_v$. Without loss of generality, assume that $a\in \varphi(A)_u$ and so $b\in \varphi(A)_v$. Then embed (we will prove in Claim 1 below that the edges can indeed be embedded) one $e$-coloured edge between the vertices $\tilde{u}$ and $\tilde{a}$ in $\Sigma$, and one   $e$-coloured edge between the vertices $\tilde{v}$ and $\tilde{b}$ in $\Sigma$.

 Secondly, suppose that exactly one of $u$ or $v$ is incident to an edge of $A$. Without loss of generality assume that $v$ is incident to an edge in $A$, then either $a \in \varphi(A)_v$ or  $b \in \varphi(A)_v$. Also without loss of generality we can assume that $b \in \varphi(A)_v$. Then embed (we will prove in Claim 1 below that the edges can indeed be embedded) one $e$-coloured edge between the vertices $\tilde{u}$ and $\tilde{a}$ in $\Sigma$. As for the vertex $u$ which is not incident to an edge in $A$, let $u'\in \V(H)$ denote the vertex given by condition three in the theorem. By the third condition in the theorem, we have $a=u'$. Embed one $e$-coloured edge between the vertices $\tilde{u}$ and $\tilde{a}$ in $\Sigma$. (The embedding can be done since  we have chosen $\{ \widetilde{G},  \widetilde{H},  \Sigma\}$ so that   $\tilde{u}$ and $\tilde{a} = \tilde{u}'$ are embedded on the same spherical component of $\Sigma$.) 
 
 Thirdly, suppose that neither $u$ and $v$ is incident to an edge of $A$. Let $u'$ and $v'$ denote the vertices of $H$ arising from $u$ and $v$, respectively, as in the third condition of the theorem.   Then also by the third condition of the theorem $a=u'$ or $b=u'$. Without loss of generality assume that  $a=u'$, then $b=v'$. In this case embed  one $e$-coloured edge between the vertices $\tilde{u}$ and $\tilde{a}$ in $\Sigma$, and one   $e$-coloured edge between the vertices $\tilde{v}$ and $\tilde{b}$ in $\Sigma$. (The embedding can be done since  we have chosen $\{ \widetilde{G},  \widetilde{H},  \Sigma\}$ so that   $\tilde{u}$ and $\tilde{a} = \tilde{u}'$ are embedded on the same spherical component of $\Sigma$, and $\tilde{v}$ and $\tilde{b}=\tilde{v'}$ are are embedded on the same spherical component of $\Sigma$.)

Note that even in the cases where $u=v$ or $a=b$, we are embedding exactly two $e$-coloured edges. Let $E$ denote the set of coloured embedded edges obtained by carrying out this procedure for each edge $e\in A^c$.  We will prove in claims 1 and 2 below that  $\{ \widetilde{G},  \widetilde{H},  \Sigma, E\}$ is a partial dual  embedding for $G$ and $H$. The sufficiency of the bijection in Theorem~\ref{t.main} will then follow by Theorem~\ref{t.pdem}.

\medskip

\noindent \underline{Claim 1:} $\{ \widetilde{G},  \widetilde{H},  \Sigma, E\}$ is a partial dual  embedding.

\medskip

\noindent {\em Proof of Claim 1:} 
By construction each edge in $E$ is incident to one vertex of $\widetilde{G}$ and one vertex of   $\widetilde{H}$, and there are exactly two edges of each colour. To show that $\{ \widetilde{G},  \widetilde{H},  \Sigma, E\}$ is a partial dual  embedding, it remains to verify that the edges in $E$ that are incident to $w$ and $\varphi(A)_w$ can actually be embedded in $\Sigma$.  To do this, let $w$ be vertex of $G$, $\tilde{w}$ be the vertex in $\widetilde{G}$ corresponding to $w$ and let $\widetilde{\varphi(A)_w}$ denote the part of the embedded graph $\widetilde{H}$ corresponding to  $\varphi(A)_w$. By Corollary~\ref{c.ed2}, if we cut $\Sigma$ along $\widetilde{\varphi(A)_w}$ then the component that contains $\tilde{w}$ is  a surface that has been obtained by identifying some of the vertices of a polygon. Furthermore,  the image of $\widetilde{G}$ in this polygon consists of $\tilde{w}$ and some half-edges between $\tilde{w}$ and the sides of the polygon. It is clear that we may then embed any number of edges between $\tilde{w}$ and the vertices of the polygon. Therefore we can embed any number of edges between  $\tilde{w}$  and  vertices of     $\widetilde{\varphi(A)_w}$ in the dual embedding $\{ \widetilde{G},  \widetilde{H},  \Sigma \}$. Finally, since any of these new embedded edges between $\tilde{w}$  and vertices of     $\widetilde{\varphi(A)_w}$ lie in the face containing $\tilde{w}$, we can embed as many edges as we please between $\tilde{w}$  and a vertex of     $\widetilde{\varphi(A)_w}$ for each  $\tilde{w}\in \V(\widetilde{G})$. Thus the set $E$ constructed above is embedded and  $\{ \widetilde{G},  \widetilde{H},  \Sigma, E\}$ is a partial dual  embedding.
This completes the proof of Claim~1.

\medskip

\noindent \underline{Claim 2:} $\{ \widetilde{G},  \widetilde{H},  \Sigma, E\}$ is a partial dual  embedding for $G$ and $H$.

\medskip

\noindent {\em Proof of Claim 2:} It remains to show that $G$ and $H$ can be obtained from  $\{ \widetilde{G},  \widetilde{H},  \Sigma, E\}$ using the procedure described in Theorem~\ref{t.pdem}.  Let $\widehat{G}$ be the graph obtained from $\widetilde{G}$ and $\widehat{H}$  be the graph obtained from $\widetilde{H}$ as described in Theorem~\ref{t.pdem}.  We first show that $G$ is isomorphic to $\widehat{G}$ ($G\cong \widehat{G}$). We immediately have that $G\backslash A^c \cong \widetilde{G} \cong \widehat{G}\backslash A^c$. Now suppose $e\in A^c$ is an edge of $G$. Then, by the construction of $\{ \widetilde{G},  \widetilde{H},  \Sigma, E\}$, $e$ incident with vertices $u$ and $v$ in $G$  if and only if there is an $e$-coloured edge between $\tilde{u}$ and $\tilde{v}$ in $\widetilde{G}$, and by the definition of $\widehat{G}$, this happens if and only if there is an $e$-coloured edge between the vertices $\hat{u}$ and $\hat{v}$  of $\widehat{G}$ that correspond to the vertices  $\tilde{u}$ and $\tilde{v}$. Thus $G\cong \widehat{G}$. A similar argument shows that $H\cong \widehat{H}$. This completes the proof of Claim~2 and of the ``if'' part of   Theorem~\ref{t.main}

\end{proof}

\bigskip

We now turn our attention to the necessity of the bijection in the statement of Theorem~\ref{t.main}.
The idea behind the proof of the ``only if'' part of Theorem~\ref{t.main} is as follows. If $G$ and $G^A$ are partial dual graphs then there is a partial dual embedding $\{G\backslash A^c ,  G^A\backslash \varphi (A^c), \Sigma, E\}$ for $G$ and $G^A$. The partial dual embedding can be used to  induce a colouring of the edges of $G^A$ using the colouring of the edges in $G$ (as described in Subsection~\ref{ss.pdg}). This defines a bijection $\varphi$ between the edge sets. Then since   $\{G\backslash A^c ,  G^A\backslash \varphi (A^c), \Sigma\}$ is a  dual embedding, it follows that $\left. \varphi\right|_{A}$ satisfies Edmonds' Criteria. Also since $\{G\backslash A^c ,  G^A\backslash \varphi (A^c), \Sigma\}$ is a  dual embedding, 
$\varphi(A)_v$ can be  identified with the edges and vertices of $ G^A\backslash \varphi (A^c)\subset \Sigma$ that bound the face containing $v$ in the dual embedding. Since each edge in  $E$ is embedded, the other end of the edge in $E$ incident to $v$ must be in  $\varphi(A)_v$.  It then follows that $\varphi$ satisfies the conditions of the theorem.

\begin{example}
Using the graphs $G$, $H=G^{\{1\}}$ and the partial dual embedding from Example~\ref{ex.pdem}. It is easily checked that the colouring of $G$ induces the colouring of $H$ given.
 Therefore $\varphi: i \mapsto i$, for $i=1,2,3$  and $\beta'=c$.

From the dual embedding we see that $(\{\alpha, \beta\}, \{1\})$ and $(\{a, b, c \}, \{1\})$ are dual graphs, $\left. \varphi\right|_{A}$ must then satisfy Edmonds' Condition with $A=\{1\}$. 
Since   $\tilde{a}$ and $\tilde{b}$ lie in faces bounded by $\    (\{\tilde{\alpha}, \tilde{\beta}\}, \{1\}) =\varphi(A)_{\alpha} $, the second condition is satisfied. Since $\tilde{\beta}$ and $\tilde{c}$ lie in a sphere with $2$- and $3$-coloured edges between them, the graphs $G$ and $H$ must be identical at the vertices $\beta$ and $c$ and the third condition of the theorem must hold.

\end{example}

\begin{proof}[Proof of Theorem~\ref{t.main} (Neccesity).]
Suppose that $G$ and $H$ are partial duals. Then $H=G^A$ for some $A\subseteq \E(G)$ and by Corollary~\ref{c.pdem}, there is a partial dual embedding $\{\widetilde{G}, \widetilde{H}, \Sigma, E\}$ for $G$ and $H$, where $\widetilde{G}=G\backslash A^c$,  $\widetilde{H}=H\backslash \varphi (A^c)$, and $A^c=\E(G)\backslash A$.
We will use this partial dual embedding to construct a mapping $\varphi: \E(G)\rightarrow \E(H)$ that satisfies the conditions of the theorem.

As described at the end of Subsection~\ref{ss.pdg}, the partial dual embedding   $\{\widetilde{G}, \widetilde{H}, \Sigma, E\}$ can be used to induce a colouring of $\E(H)$ using the colouring of  $\E(G)$. Assume that $\E(H)$  is equipped with this induced colouring. Define 
a mapping $\varphi: \E(G)\rightarrow \E(H)$ by setting $\varphi(e)$ to be the unique edge in $\E(H)$ of the same colour as $e$.  It remains to show that $\varphi$ satisfies the conditions stated in the Theorem~\ref{t.main}. This is verified in the sequence  of four claims below.

\medskip

\noindent \underline{Claim 1:} $\varphi$ is a bijection.

\medskip

\noindent {\em Proof of Claim 1:} 
Claim 1 follows immediately from the definition of $\varphi$.

\noindent \underline{Claim 2:} $\left. \varphi\right|_{A}$ satisfies Edmonds' Criteria.

\medskip

\noindent {\em Proof of Claim 2:} 
$\left. \varphi\right|_{A}:  A \rightarrow \varphi(A) $. Recall that   
 $\widetilde{G}\cong G\backslash A^c$,  $\widetilde{H} \cong H\backslash \varphi (A^c)$ and $\{\widetilde{G}, \widetilde{H}, \Sigma\}$ is a dual embedding. For each  $e\in A$, $\varphi(e)$ is defined to be the edge of $\ \varphi(A)$  corresponding to the edge in  $\widetilde{H}$ that intersects the  $e$-coloured in $\widetilde{G}$. It then follows by Lemma~\ref{l.ed} that $\left. \varphi\right|_{A}$ satisfies Edmond's criteria. This completes the proof of Claim~2.

\noindent \underline{Claim 3:} If $e\in \E(G)$ is incident to a vertex $v\in \V(G)$ and $v$ is incident to an edge in $A$, then $\varphi(e)$ is incident to a vertex of $\varphi(A)_v$.  Moreover, if both ends of $e$ are incident to $v$, then both ends of $\varphi(e)$ are incident to vertices of $\varphi(A)_v$.

\medskip

\noindent {\em Proof of Claim 3:} 
There are two cases to consider: when $e\in A$ and when $e\notin A$.

\medskip

First suppose that $e\in A$ and that $e$ is incident to a vertex $v$. Then $\varphi(A)_v$ is the subgraph of $H$ induced by the images of $A$ that are incident to $v$. However, $e\in A$ and $e$ is incident to $v$ so $\varphi(e)$ is an edge  in the subgraph $\varphi(A)_v$ and is therefore incident to a vertex of  $\varphi(A)_v$ as required.  

\medskip

Now suppose that $e\notin A$ and that $e$ is incident to a vertex $v$. Let $\tilde{v}$ denote the vertex in $\widetilde{G}$ corresponding to $v$. Then there is an embedded $e$-coloured edge $\tilde{f}$ in $E$ between $\tilde{v}$ and a vertex $\tilde{a}$ of $\widetilde{H}$. We need to show that $\tilde{a}$ corresponds to a vertex in $\varphi(A)_v$.

Consider the partial dual embedding $\{\widetilde{G}, \widetilde{H}, \Sigma, E\}$. In this embedding, the vertex $\tilde{v}$ lies in a face of the embedded graph  $\widetilde{H}\subset \Sigma$. Let $S$ denote the subgraph of  $\widetilde{H}$ that bound this face. Since the edge  $\tilde{f}\in E$ is embedded, it is incident to a vertex in $S$. Thus $\tilde{a}\in S$.
Finally, since $\{\widetilde{G}, \widetilde{H}, \Sigma\}$ is a dual embedding and  $\widetilde{G}=G\backslash A^c$,  the edges in $S$ correspond to the edges in $\varphi(A)_v$ and therefore $\tilde{a}$ corresponds to a vertex in $\varphi(A)_v$ as required. 
This completes the proof of Claim~3. 

\noindent \underline{Claim 4:} If $v\in \V(G)$ is not incident to an edge in $A$, then there exists a vertex $v' \in \V(H)$ with the property that  $e\in \E(G)$ is incident to $v$ if and only if  $\varphi(e)\in \E(H)$ is incident to $v'$. Moreover, both end of $e$ incident to $v$ if and only if both ends of $\varphi(e)$ are incident to $v'$.

\medskip

\noindent {\em Proof of Claim 4:} 
 $v\in \V(G)$ is not incident to an edge in $A$. Then there is a spherical component $S^2$ of $\Sigma$ which contains the vertex $\tilde{v}$ of $\widetilde{G}$ that corresponds to $v$. $S^2$ contains no other vertices of  $\widetilde{G}$.    Since $\{\widetilde{G}, \widetilde{H}, \Sigma, E\}$ is a partial dual embedding, this component also contains a single vertex $\tilde{a} \in \V(\widetilde{H})$ and some embedded edges in $E$ between the vertices.  The vertex in $H$ corresponding to $\tilde{a}$ is the vertex $v'$ in the statement. Then by the construction of $\varphi$, it is clear that an end of an edge $e\in \E(G)$ is incident to $v$ precisely when a unique end of $\varphi(e)\in \E(H)$ is incident to $v'$, as required.
This completes the proof of Claim~4 and of Theorem~\ref{t.main}.
\end{proof}

\end{document}